\documentclass[a4paper,reqno]{amsart}

\usepackage{amssymb}
\usepackage{latexsym}
\usepackage{amsmath}
\usepackage{amsthm}
\usepackage{euscript}
\usepackage{bbm}
\usepackage{tikz}
\usepackage{float}
\usetikzlibrary{automata,positioning,arrows.meta}
\usepackage{pgfplots}
\usepgfplotslibrary{polar}
   \def\cB{{\mathcal B}}   
   \def\cE{{\mathcal E}}   
\def\cG{{\mathcal G}}   \def\cH{{\mathcal H}}   
      
   \def\cN{{\mathcal N}}

\def\cV{{\mathcal V}}

\def\cal H{{\mathcal H}}

\def\R{\mathbb{R}}
\def\C{\mathbb{C}}

\def\N{\mathbb{N}}

\def\ran{{\text{\rm ran\,}}}
\def\dom{{\text{\rm dom\,}}}

\def\phi{\varphi}
\def\dd{\textup{d}}

\DeclareMathOperator{\spann}{span}

\renewcommand{\theta}{\vartheta}

\newtheorem{theorem}{Theorem}[section]
\newtheorem*{thm*}{Theorem}
\newtheorem{proposition}[theorem]{Proposition}
\newtheorem{corollary}[theorem]{Corollary}
\newtheorem{lemma}[theorem]{Lemma}

\theoremstyle{definition}
\newtheorem{definition}[theorem]{Definition}
\newtheorem{example}[theorem]{Example}

\newtheorem{hypothesis}[theorem]{Hypothesis}

\newtheorem{remark}[theorem]{Remark}
\newtheorem*{ack}{Acknowledgement}

\numberwithin{equation}{section}

\title[The Krein--von Neumann extension for metric graphs]{The Krein--von Neumann extension for Schr\"odinger operators on metric graphs}

\author[J.~Muller]{Jacob Muller}
\email{muller@math.su.se}

\author[J.~Rohleder]{Jonathan Rohleder}
\email{jonathan.rohleder@math.su.se}
\address{Matematiska institutionen\\ Stockholms universitet \\ 106 91 Stockholm \\
Sweden}

\dedicatory{Dedicated with great pleasure to Henk de Snoo on the occasion of his 75th birthday}

\begin{document}

\begin{abstract}
The Krein--von Neumann extension is studied for Schr\"odinger operators on metric graphs. Among other things, its vertex conditions are expressed explicitly, and its relation to other self-adjoint vertex conditions (e.g.\ continuity-Kirchhoff) is explored. A variational characterisation for its positive eigenvalues is obtained. Based on this, the behaviour of its eigenvalues under perturbations of the metric graph is investigated, and so-called surgery principles are established. Moreover, isoperimetric eigenvalue inequalities are obtained.
\end{abstract}

\maketitle

\section{Introduction}

It is an almost hundred-year-old story that many of the differential operators appearing in mathematical physics and their boundary conditions can be described conveniently in the framework of extension theory of symmetric operators. A complete description of all self-adjoint extensions of a symmetric operator was first given by von Neumann \cite{N30}. On the other hand, it turned out that a theory of self-adjoint extensions of symmetric operators that are semibounded from below can be done conveniently by means of semibounded sesquilinear forms; this originates from the work of Friedrichs \cite{F34}. However, it is due to Krein \cite{K47} (see also the works of Vishik \cite{V52} and Birman \cite{B56}) that among all non-negative extensions of a positive definite symmetric operator $S$, there are two extremal ones, the Friedrichs extension $S_{\rm F}$ and the (by now so-called) Krein--von Neumann extension $S_{\rm K}$, in the sense that each non-negative self-adjoint extension $A$ of $S$ satisfies
\begin{align*}
 S_{\rm K} \leq A \leq S_{\rm F}.
\end{align*}
These inequalities may be understood in the sense of quadratic forms or via the involved operators' resolvents. It is beyond the scope of this article to provide a complete historical review of the developments related to the Krein--von Neumann extension; for further reading we refer the reader to \cite{AN70} and the survey articles \cite{AS80,AGMST13}. Among the abstract advancements on extremal extensions of positive definite symmetric operators (and, more generally, symmetric linear relations), we mention \cite{AHSS02,AT02,AT03,CS78,HMS04,M92,PS96,SS03,S96,T81}.

In the study of e.g.\ elliptic second order differential operators on Euclidean domains, the Friedrichs extension is a very natural object; for instance, for the minimal symmetric Laplacian on a bounded domain in $\R^n$ corresponding to both Dirichlet and Neumann boundary conditions, the Friedrichs extension is the self-adjoint Laplacian subject to Dirichlet boundary conditions. On the other hand, in the same setting, the Krein--von Neumann extension corresponds to certain non-local boundary conditions which can be described in terms of the associated Dirichlet-to-Neumann map; for properties of the Krein--von Neumann extension of elliptic differential operators and recent related developments, we refer the reader to \cite{AK11,AGMT10,BGMM16,G83,G12,M94}.

For differential operators on metric graphs, which we consider in the present paper, the situation is similar, yet different in some respects. If $\Gamma$ is a finite metric graph, then we take, as a starting point, the (negative) Laplacian (i.e.\ the negative second derivative operator on each edge) $S$ in $L^2 (\Gamma)$, which satisfies on each vertex both Dirichlet and Kirchhoff vertex conditions; that is, the functions in the domain of $S$ vanish and have derivatives which sum up to zero at each vertex. This symmetric operator is very natural to carry out extension theory, since its adjoint $S^*$ is the Laplacian on $\Gamma$ with continuity as its (only) vertex conditions. Therefore any self-adjoint extension of $S$ in $L^2 (\Gamma)$ (which, at the same time is a restriction of $S^*$) satisfies continuity conditions and thus reflects, at least to some extent, the connectivity of the graph. Nevertheless, the Friedrichs extension of $S$ in this setting is the Laplacian on functions which are zero at every single vertex, an operator which, despite continuity, is determined by the graph's edges considered as separate intervals, instead of the actual graph structure. Other self-adjoint extensions of $S$ which are more suitable for the spectral analysis of network structures are the operator with continuity-Kirchhoff conditions (the so-called standard or natural Laplacian) or with $\delta$-type vertex conditions; the latter we will not discuss here further. 

The Krein--von Neumann extension for a Laplacian on a metric graph has not been considered much in the literature so far; an attempt for a symmetric operator with vertex conditions different from the ones considered here was done in \cite{M15}. The Krein--von Neumann extension of our operator $S$ is, like for the minimal Laplacian on a Euclidean domain, an operator with non-local vertex conditions. Nevertheless, its domain is intimately connected to the structure of the underlying graph. In fact, we prove that the matrix that couples the values and the sums of derivatives at the vertices for functions in the domain of $S_{\rm K}$ is exactly the weighted discrete Laplacian on the underlying discrete graph, where the weights are the inverse edge lengths. 

Our main focus in the present paper is on spectral properties of the operator $S_{\rm K}$, not only in the case of the Laplacian, but also for Schr\"odinger operators with nonnegative potentials $q_e$ on the edges. Namely, we consider the operator $S$ acting as $- \frac{\dd^2}{\dd x^2} + q_e$ on each edge $e$ of $\Gamma$, with Dirichlet and Kirchhoff vertex conditions as described above in the case of the Laplacian. Its Krein--von Neumann extension, the so-called {\em perturbed Krein Laplacian}, denoted by $- \Delta_{{\rm K}, \Gamma, q}$, is the main object of consideration in this article. We first describe the domain of $- \Delta_{{\rm K}, \Gamma, q}$ in terms of vertex conditions and establish Krein-type formulae for the resolvent differences with both the Friedrichs extension (the Schr\"odinger operator with Dirichlet vertex conditions) and the the Schr\"odinger operator $- \Delta_{{\rm st}, \Gamma, q}$ with standard vertex conditions. As a consequence, we obtain the formula
\begin{align*}
 \dim \ran \Big[ \big(- \Delta_{{\rm K}, \Gamma, q} - \lambda \big)^{-1} & - \big(- \Delta_{{\rm st}, \Gamma, q} - \lambda \big)^{-1} \Big] = \begin{cases} V - 1 & \text{if}~q = 0~\text{identically}, \\ V, & \text{else}, \end{cases}
\end{align*}
in which $V$ denotes the number of vertices of $\Gamma$ and $\lambda$ takes appropriate complex values. This formula distinguishes the potential-free case clearly from the case influenced by a potential. It also sheds light on another interesting phenomenon: the Krein Laplacian may, in some rare occasions, coincide with the standard Laplacian, and this is the case if and only if $\Gamma$ has only one vertex (with possibly many loops attached to it) and thus is a so-called flower graph. Moreover, we use the Krein-type resolvent formulae to obtain some results on spectral asymptotics of the perturbed Krein Laplacian. 

A further property of the perturbed Krein Laplacian on a metric graph $\Gamma$, which we establish, is the possibility to describe its positive eigenvalues variationally. In fact, the spectrum of $- \Delta_{{\rm K}, \Gamma, q}$ is purely discrete, and the lowest eigenvalue is always zero, with multiplicity equal to $V$, the number of vertices, as we show. All its positive eigenvalues $\lambda_j^+ (- \Delta_{{\rm K}, \Gamma, q})$, ordered nondecreasingly and counted with multiplicities, can be characterised by the variational principle
\begin{align}\label{eq:minMaxIntro}
 \lambda_j^+ \big( - \Delta_{{\rm K}, \Gamma, q} \big) = \min_{\substack{F \subset \widetilde H_0^2 (\Gamma) \\ \dim F = j}} \max_{\substack{f \in F \\ f \neq 0}} \frac{\int_\Gamma \left|- f'' + q f\right|^2 \dd x}{\int_\Gamma |f'|^2 \dd x + \int_\Gamma q |f|^2 \dd x};
\end{align}
here, $\widetilde H^2_0 (\Gamma)$ is the second-order Sobolev space on each edge, equipped with Dirichlet and Kirchhoff conditions at every vertex (see \eqref{eq:SV}). This formula is the exact counterpart of a variational description of the positive eigenvalues of the perturbed Krein Laplacian on a domain in $\R^n$, which was established in \cite[Proposition 7.5]{AGMT10}. Before we derive~\eqref{eq:minMaxIntro}, we first establish an abstract version of this principle; see Theorem \ref{thm:minMaxKrein}. Its proof is along the lines of the result for the Laplacian in \cite{AGMT10}; however, we found it useful and of independent interest to have it at hand also abstractly for the Krein--von Neumann extension of any symmetric, positive definite operator $S$ for which $\dom S$ equipped with the graph norm of $S$ satisfies a compactness condition. As a consequence of the formulation for graphs \eqref{eq:minMaxIntro}, we easily obtain inequalities between the (positive) eigenvalues of the perturbed Krein Laplacian and other self-adjoint extensions of $S$. 

An important field of application of the eigenvalue characterisation \eqref{eq:minMaxIntro} are so-called surgery principles. Such principles study the influence of geometric perturbations of a metric graph on the spectra of associated Laplacians or more general differential operators. The reader may think of sugery operations such as joining two vertices into one or cutting through a vertex, or adding or removing edges (or even entire subgraphs). Such principles were studied in depth for the Laplacian or Schr\"odinger operators subject to standard (and some other local) vertex conditions; see \cite{BKKM19,KKT16,KKMM16,KN14,RS20}. As we point out, the eigenvalues of the perturbed Krein Laplacian behave in some respects in the same way as the eigenvalues of $- \Delta_{{\rm st}, \Gamma, q}$; for instance, when gluing vertices all eigenvalues increase (or stay the same), and adding pendant edges or graphs (a process which increases the ``volume'' of $\Gamma$) may only decrease the eigenvalues. On the other hand, in some respects the behaviour is different from what we are used to for standard vertex conditions. Let us only mention three examples: firstly, for the positive eigenvalues, gluing vertices has actually a non-increasing effect (but at the same time also the multiplicity of the eigenvalue 0 decreases), whilst for standard vertex conditions, the positive eigenvalues behave non-decreasingly and the dimension of the kernel remains the same. Secondly, removing a vertex of degree two (replacing the two incident edges by one) may change eigenvalues in a monotonous way, whilst it does not have any influence on the spectrum of an operator with standard vertex conditions. Thirdly, inserting an edge between two existing vertices makes all eigenvalues decrease (or stay the same); for standard vertex conditions, this is not necessarily the case; see e.g. \cite{KMN13}.

A typical application of surgery principles for graph eigenvalues consists of deriving spectral inequalities in terms of geometric and topological parameters of the graph such as its total length, diameter, number of edges or vertices, or its first Betti number (or Euler characteristics, equivalently). For a few recent advances on spectral inequalities for quantum graphs, we refer to \cite{BL17,BKKM17,K20,KN19,MP20,P20}. To demonstrate how surgery principles for the perturbed Krein Laplacian on a metric graph may be applied, we establish lower bounds for the positive eigenvalues, in terms of eigenvalues of a loop graph or edge lengths. For instance, for the first positive eigenvalue of the Krein Laplacian without potential the lower bound is explicit,
\begin{align*}
 \lambda_1^+ (- \Delta_{\rm K, \Gamma}) \geq 4 \left( \frac{\pi}{\ell (\Gamma)} \right)^2,
\end{align*}
where $\ell (\Gamma)$ denotes the total length of $\Gamma$, and we specify the class of graphs for which this estimate is optimal.

Considering the Krein--von Neumann (and other) extensions of a Schr\"odinger operator with Dirichlet and Kirchhoff vertex conditions at all vertices is natural, as we pointed out above. However, it may also be useful to study extensions of a symmetric Schr\"odinger operator with different vertex conditions. We mention, as an example, the Laplacian with both Dirichlet and Neumann (Kirchhoff) vertex conditions at the ``loose ends'', i.e.\ the vertices of degree one, but standard vertex conditions at all interior vertices. In this case, the vertex conditions of the Krein--von Neumann extension will still be standard at all interior vertices, but they will couple the vertices of degree one in a nonlocal way. We conclude our paper with a short section where we discuss such situations.

Let us briefly describe how this paper is organised. In Section \ref{sec:abstractKrein}, we review some background on the abstract Krein--von Neumann extension. Moreover, we provide a proof of the abstract counterpart of the variational principle \eqref{eq:minMaxIntro} and derive a few easy consequences. Additionally, we study some basic properties of boundary triples, which we use as a tool. The aim of Section \ref{sec:Krein} is to introduce the perturbed Krein Laplacian on a metric graph and to study its properties, such as a description of its domain, Krein-type resolvent formulae and some consequences of the min-max principle. Section \ref{sec:surgery} is devoted to a collection of surgery principles, whilst in Section \ref{sec:isoperimetric}, we apply some of them in order to obtain some isoperimetric inequalities. Finally, Section \ref{sec:moreGeneral} deals with the more general setting where self-adjoint vertex conditions are fixed at some vertices, and extension theory is applied with respect to the remaining vertices.

\section{The abstract Krein--von Neumann extension and its eigenvalues}\label{sec:abstractKrein}

\subsection{Preliminaries}

Throughout this section we assume that $\cH$ is a separable complex Hilbert space with inner product $(\cdot, \cdot)$ and corresponding norm $\| \cdot \|$. For any closed linear operator $A$ in $\cH$, we denote by $\sigma (A)$ and $\rho (A)$ its spectrum and resolvent set respectively. If $A$ is self-adjoint and has a purely discrete spectrum bounded from below, then we write
\begin{align*}
 \lambda_1 (A) \leq \lambda_2 (A) \leq \dots
\end{align*}
for its eigenvalues, counted according to their multiplicities. If $\cG$ is a further Hilbert space, we denote by $\cB (\cG, \cH)$ the space of all bounded, everywhere-defined linear operators from $\cG$ to $\cH$ and abbreviate $\cB (\cG) := \cB (\cG, \cG)$.

We make the following assumption.

\begin{hypothesis}\label{hyp:abs}
The operator $S : \cH \supset \dom S \to \cH$ is closed and symmetric with dense domain $\dom S$. Furthermore, $S$ has a positive lower bound, i.e.\ there exists $\mu > 0$ such that
\begin{align}\label{eq:mu}
 (S f, f) \geq \mu \|f\|^2, \qquad f \in \dom S.
\end{align}
\end{hypothesis}

Under Hypothesis \ref{hyp:abs}, the defect numbers $(n_-, n_+)$ of $S$ satisfy $n_- = n_+ = \dim \ker S^*$, where $S^*$ denotes the adjoint of $S$.
% To avoid talking about trivialities we assume that $S$ is not self-adjoint, that is, 
Moreover, it follows directly from \eqref{eq:mu} that $\dom S \cap \ker S^* = \{0\}$, and the Krein--von Neumann extension of $S$ can be defined as follows. 

\begin{definition}
The {\em Krein--von Neumann extension} of $S$ is the operator $S_{\rm K}$ in $\cH$ given by
\begin{align}\label{eq:AK}
 S_{\rm K} f = S^* f, \qquad \dom S_{\rm K} = \dom S \dotplus \ker S^*.
\end{align}
\end{definition}

It is well-known that $S_{\rm K}$ is self-adjoint and is the smallest non-negative self-adjoint extension of $S$ in the sense of quadratic forms. Its counterpart, the Friedrichs extension of $S$, is the largest non-negative extension of $S$ and we denote it by $S_{\rm F}$. It can be defined via completion of the quadratic form induced by $S$; we do not go into the details but refer the reader to, e.g. the discussion in \cite[Chapter~VI]{K95}. For any self-adjoint, non-negative extension $A$ of $S$, the relation
\begin{align*}
 \big( (S_{\rm F} - \lambda)^{-1} f, f \big) \leq \big( (A - \lambda)^{-1} f, f \big) \leq \big( (S_{\rm K} - \lambda)^{-1} f, f \big), \quad f \in \cH,
\end{align*}
holds for each $\lambda < 0$. The spectrum of the Friedrichs extension has a strictly positive lower bound; in fact, $\min \sigma (S_{\rm F})$ coincides with the supremum over all $\mu$ such that \eqref{eq:mu} holds. Conversely, the Krein--von Neumann extension $S_{\rm K}$ has the point 0 as the bottom of its spectrum, and the corresponding eigenspace is given by
\begin{align*}
 \ker S_{\rm K} = \ker S^*,
\end{align*}
which follows from the definition of $S_{\rm K}$ and the fact that 0 is not an eigenvalue of $S$. In particular, $\dim \ker S_{\rm K} = n_- = n_+$, the defect number of $S$. We refer the reader to, e.g.\ the survey \cite{AGMST13} for a more detailed discussion of the Krein--von Neumann extension.

\subsection{A variational characterisation of the positive eigenvalues of the Krein--von Neumann extension}

The main goal of this subsection is to provide an abstract variational description of the eigenvalues different from 0 of the Krein--von Neumann extension. The credits for the arguments that lead to the min-max principle in Theorem~\ref{thm:minMaxKrein} below go to the articles \cite{AGMST10,AGMST13,AGMT10}, where the abstract Krein--von Neumann extension and the perturbed Krein Laplacian on domains in $\R^n$ were studied. There, the min-max principle is stated in the context of the application, so for the convenience of the reader we state and prove this variational principle here abstractly.

Associated with the operator $S$ is the space 
\begin{align*}
 \cH_S := \dom S \qquad \text{with norm} \qquad \|f\|_S := \|S f\|, \quad f \in \cH_S.
\end{align*}
Due to~\eqref{eq:mu}, $\cH_S$ is a normed space, and as $S$ is closed, it follows that $\cH_S$ is a Banach space. The norm $\| \cdot \|_S$ corresponds to the inner product $(f, g)_S = (S f, S g)$; hence $\cH_S$ is a Hilbert space. Moreover, there exists a constant $\widetilde \mu > 0$ such that 
\begin{align}\label{eq:alsoDoch}
 \|f\|_S \geq \widetilde \mu \|f\|, \qquad f \in \cH_S.
\end{align}
(Indeed, if not then for each $n \in \N$ there exists $f_n \in \cH_S$, w.l.o.g.\ $\|f_n\| = 1$, such that $\|S f_n\| < \frac{1}{n}$ and hence $\mu \leq (S f_n, f_n) \leq \|S f_n\| < \frac{1}{n}$ by~\eqref{eq:mu}, a contradiction to $\mu > 0$.) We further denote by $\cH_S^*$ the dual space of $\cH_S$ and write $(\cdot, \cdot)_{\cH_S^*, \cH_S}$ for the sesquilinear duality between $\cH_S^*$ and $\cH_S$, i.e.\ the continuous extension of
\begin{align*}
 (h, f)_{\cH_S^*, \cH_S} := (h, f), \qquad h \in \cH, \quad f \in \cH_S,
\end{align*}
to all $h \in \cH_S^*$. (Note that $\cH$ is dense in $\cH_S^*$ as $\cH_S$ is dense in $\cH$.)

It will sometimes be useful to consider $S$ as an operator from $\cH_S$ to $\cH$ rather than as an operator in $\cH$. Therefore we define
\begin{align*}
  \widetilde S : \cH_S \to \cH, \qquad \widetilde S f : = S f, \quad f \in \cH_S.
\end{align*}
Then $\widetilde S$ is bounded and its adjoint $\widetilde S^*$ is the unique bounded operator from $\cH$ to $\cH_S^*$ that satisfies
\begin{align*}
 \big( \widetilde S f, g \big) = \big( f, \widetilde S^* g \big)_{\cH_S, \cH_S^*}, \qquad f \in \cH_S, g \in \cH.
\end{align*}
Note that on the left-hand side we might as well replace $\widetilde S$ by $S$. For later use, we remark also that $\widetilde S^* g \in \cH$ implies $g \in \dom S^*$ and $S^* g = \widetilde S^* g$. 
%Indeed, if $\widetilde S^* f \in \cH$ then
%\begin{align*}
 %(\widetilde S^* f, g) = (\widetilde S^* f, g)_{\cH_S^*, \cH_S} = (f, \widetilde S g) = (f, S g)
%\end{align*}
%holds for all $g \in \cH_S$, which implies the observation. 
In particular,
\begin{align}\label{eq:Kernels!}
 \ker \widetilde S^* = \ker S^*.
\end{align}

The following lemma is a variant of~\cite[Lemma~3.1]{AGMST10}. For the convenience of the reader, we provide a complete proof.

\begin{lemma}\label{lem:B}
Let Hypothesis \ref{hyp:abs} be satisfied. Then the operator $\widetilde S^* S : \cH_S \to \cH_S^*$ is bijective, and
\begin{align}\label{eq:B}
 B := (\widetilde S^* S)^{- 1} S : \cH_S \to \cH_S
\end{align}
is a bounded, self-adjoint, non-negative operator with $\ker B = \{0\}$. Moreover, a number $\lambda > 0$ is an eigenvalue of $S_{\rm K}$ if and only if $\lambda^{- 1}$ is an eigenvalue of $B$.
\end{lemma}

\begin{proof}
The operator $\widetilde S^* S$ is injective as $\widetilde S^* S f = 0$ implies 
\begin{align*}
 0 = \big( \widetilde S^* S f, f \big)_{\cH_S^*, \cH_S} = (S f, S f) = \|S f\|^2,
\end{align*}
that is, $f \in \ker S$ which by~\eqref{eq:mu} implies $f = 0$. Furthermore, let $h \in \cH_S^*$. According to the Fr\'echet--Riesz theorem, there exists a unique $f \in \cH_S$ such that
\begin{align*}
 (g, h)_{\cH_S, \cH_S^*} = (g, f)_S = (S g, S f) = \big( g, \widetilde S^* S f \big)_{\cH_S, \cH_S^*} 
\end{align*}
holds for all $g \in \cH_S$, and hence $\widetilde S^* S f = h$. Thus $\widetilde S^* S$ is bijective and, by the open mapping theorem, has a bounded inverse. In particular, the operator $B$ in~\eqref{eq:B} is well-defined and bounded as the product of two bounded operators. 

Let us show next that $B$ is symmetric and thus self-adjoint. Indeed, for $f \in \cH_S$, we get
\begin{align}\label{eq:Bsym}
\begin{split}
 (B f, f)_S & = \big( S ( \widetilde S^* S)^{- 1} S f, S f \big) = \big( \widetilde S^* S ( \widetilde S^* S)^{- 1} S f, f \big)_{\cH_S^*, \cH_S} \\
 & = (S f, f) \geq \mu \|f\|^2
\end{split}
\end{align}
by~\eqref{eq:mu} and, in particular, $(B f, f)_S \in \R$. Hence $B$ is self-adjoint and non-negative, and~\eqref{eq:Bsym} also implies that $\ker B = \{0\}$.

Now let $\lambda > 0$ be such that $S_{\rm K} g = \lambda g$ holds for some $g \in \dom S_{\rm K}$, $g \neq 0$. Define also $f := S_{\rm F}^{- 1} S_K g$, where $S_{\rm F}$ is the Friedrichs extension of $S$. As $0 \notin \sigma (S_{\rm F})$ by~\eqref{eq:mu}, $f$ is well-defined and belongs to $\dom S_{\rm F}$. Moreover, as $g \in \dom S_{\rm K}$, by~\eqref{eq:AK} we can write $g = g_S + g_*$ with $g_S \in \dom S$ and $g_* \in \ker S^*$ and get
\begin{align}\label{eq:klapptDoch}
 f = S_{\rm F}^{- 1} S_{\rm K} g = S_{\rm F}^{- 1} S g_S + S_{\rm F}^{- 1} S^* g_* = S_{\rm F}^{- 1} S_{\rm F} g_S = g_S \in \dom S.
\end{align}
Furthermore, $f \neq 0$ as otherwise $g \in \ker S_{\rm K}$, contradicting $S_{\rm K} g = \lambda g \neq 0$, and $S f = S_{\rm F} f = S_{\rm K} g = \lambda g$ together with~\eqref{eq:klapptDoch} yields
\begin{align*}
 \widetilde S^* S f = \lambda \widetilde S^* g = \lambda S^* (g_S + g_*) = \lambda S g_S = \lambda S f.
\end{align*}
Thus $B f = \lambda^{- 1} f$, that is, $\lambda^{- 1}$ is an eigenvalue of $B$. 

Conversely, let $B f = \lambda^{- 1} f$ for some $\lambda > 0$ and $f \in \cH_S$, $f \neq 0$. Then $\widetilde S^* S f = \lambda S f$, which can be rewritten as $\widetilde S^* (S - \lambda) f = 0$, that is, $(S - \lambda) f \in \ker \widetilde S^* = \ker S^*$; see~\eqref{eq:Kernels!}. Define $g := \lambda^{- 1} S f$. Then $g$ is nonzero and
\begin{align*}
 f + \lambda^{- 1} (S - \lambda) f = f + g - f = g,
\end{align*}
which, due to $f \in \dom S$ and $(S - \lambda) f \in \ker S^*$, implies $g \in \dom S_{\rm K}$. Finally,
\begin{align*}
 S_{\rm K} g = \lambda^{- 1} \widetilde S^* S f = S f = \lambda g,
\end{align*}
that is, $\lambda$ is an eigenvalue of $S_{\rm K}$. This completes the proof.
\end{proof}

We point out that Lemma~\ref{lem:B} describes, in an abstract setting, the coincidence between the positive eigenvalues of the Krein--von Neumann extension and the eigenvalues of an abstract buckling problem; the latter reads $\widetilde S^* S f = \lambda S f$ and is discussed in detail in~\cite[Section~3]{AGMST10}.

Next we provide an abstract version of the min-max principle established for Krein Laplacians on domains in~\cite[Proposition~7.5]{AGMT10}. The Rayleigh quotient
\begin{align*}
 R_{\rm K} [f] := \frac{\|S f\|^2}{(S f, f)}, \qquad f \in \dom S, f \neq 0,
\end{align*}
is well-defined due to~\eqref{eq:mu}.

\begin{theorem}\label{thm:minMaxKrein}
Assume that Hypothesis \ref{hyp:abs} is satisfied and that the embedding $\iota : \cH_S \to \cH$ is compact. Then $\sigma (S_{\rm K})$ is nonnegative and consists of isolated eigenvalues, and the positive eigenvalues have finite multiplicities. Moreover, counted with multiplicities, the positive eigenvalues $\lambda_1^+ (S_{\rm K}) \leq \lambda_2^+ (S_{\rm K}) \leq \dots$ of $S_{\rm K}$ satisfy
\begin{align*}
 \lambda_j^+ (S_{\rm K}) = \min_{\substack{F \subset \dom S \\ \dim F = j}} \max_{\substack{f \in F \\ f \neq 0}} R_{\rm K} [f]
\end{align*}
for all $j \in \N$. 
\end{theorem}

\begin{proof}
As the embedding $\iota : \cH_S \to \cH$ is compact, it follows that the Friedrichs extension $S_{\rm F}$ of $S$ has a compact resolvent, from which it can be deduced that the compression of $S_{\rm K}$ to $(\ker S_{\rm K})^\perp$ has purely discrete spectrum equal to $\sigma (S_{\rm K}) \setminus \{0\}$; see, e.g.,~\cite[Theorem~2.10]{AGMST13}. In particular, the eigenvalues of $S_{\rm K}$ cannot accumulate to zero and thus $\sigma (S_{\rm K})$ is a discrete set, and the positive eigenvalues have finite multiplicities.

For the rest of this proof, we make the abbreviation $\lambda_j := \lambda_j^+ (S_{\rm K})$. Let $B : \cH_S \to \cH_S$ be the bounded, self-adjoint, nonnegative operator in Lemma~\ref{lem:B} whose eigenvalues coincide with $\{\lambda_j^{- 1} : j \in \N \}$. As $\iota$ is compact, the same holds for the embedding $\iota^* : \cH \to \cH_S^*$, and $B$ can be rewritten as
\begin{align*}
 B = (\widetilde S^* S)^{- 1} \iota^* S,
\end{align*}
which is also then compact. In particular, we can choose an orthonormal basis $\{f_j : j \in \N \}$ of $\cH_S$ such that $\lambda_j B f_j = f_j$, or equivalently $S^* S f_j = \lambda_j S f_j$, holds for all $j \in \N$. (Here we are assuming $\dim \cH_S = \infty$; the finite-dimensional case is exactly the same with a finite orthonormal basis.) Then for each $j \in \N$,
\begin{align*}
 R_{\rm K} [f_j] = \frac{\|S f_j\|^2}{(S f_j, f_j)} = \lambda_j \frac{\|S f_j\|^2}{(S^* S f_j, f_j)} = \lambda_j
\end{align*}
holds. Let us define $F_0 := \{0\}$ and
\begin{align*}
 F_j := \spann \left\{ f_k : k \leq j \right\}, \qquad j = 1, 2, \dots,
\end{align*}
and denote by $F_{j - 1}^\perp$ the orthogonal complement of $F_{j - 1}$ with respect to the inner product $(\cdot, \cdot)_S$ in $\cH_S$ for all $j \in \N$. Now fix $j \in \N$. Then any $f \in F_{j - 1}^\perp$ can be written as $f = \sum_{k = j}^\infty c_k f_k$ for appropriate $c_k \in \C$, where the sum converges in $\cH_S$ (and hence also in $\cH$ due to~\eqref{eq:alsoDoch}). Then the continuity of $S$ with respect to the norm in $\cH_S$ implies
\begin{align*}
 (S f, f) & = \sum_{k = j}^\infty c_k (S f_k, f) = \sum_{k = j}^\infty \lambda_k^{- 1} c_k (S^* S f_k, f) = \sum_{k = j}^\infty \lambda_k^{- 1} |c_k|^2 (f_k, f_k)_S \\
 & \leq \lambda_j^{- 1} \|f\|_S^2,
\end{align*}
and thus $R_{\rm K} [f] \geq \lambda_j$ for all $f \in F_{j - 1}^\perp$, with equality for $f = f_j$. Consequently,
\begin{align}\label{eq:minV}
 \min_{\substack{f \in F_{j - 1}^\perp \\ f \neq 0}} R_{\rm K} [f] = \lambda_j.
\end{align}
By a similar calculation, one verifies
\begin{align}\label{eq:maxV}
 \max_{\substack{f \in F_j \\ f \neq 0}} R_{\rm K} [f] = \lambda_j. 
\end{align}
Now let $G_j \subset \cH_S$ be a $j$-dimensional subspace with $G_j \neq F_j$. Then by a dimension argument, there exists $g_j \in (G_j \cap F_{j - 1}^\perp) \setminus \{0\}$, and~\eqref{eq:minV} gives
\begin{align*}
 \lambda_j = \min_{\substack{f \in F_{j - 1}^\perp \\ f \neq 0}} R_{\rm K} [f] \leq R_{\rm K} [g_j] \leq \max_{\substack{f \in G_j \\ f \neq 0}} R_{\rm K} [f].
\end{align*}
Together with~\eqref{eq:maxV}, this implies the assertion of the theorem.
\end{proof}

As a direct consequence, one gets the following comparison principle for the positive eigenvalues of $S_{\rm K}$ and the eigenvalues of any self-adjoint extension of $S$. The inequality between eigenvalues of $S_{\rm F}$ and $S_{\rm K}$ is mentioned for completeness, but it has been known for a long time, see, e.g. \cite[Theorem~5.1]{AS80}. However, it follows conveniently from the above min-max principle.

\begin{theorem}\label{thm:inequalitiesAbstract}
Assume that Hypothesis \ref{hyp:abs} is satisfied and that the embedding $\iota : \cH_S \to \cH$ is compact, and let $A$ be any self-adjoint extension of $S$ with a purely discrete spectrum. Moreover, let $d := \dim \ker A$. Then
\begin{align}\label{eq:indexShift}
 \lambda_{j + d} (A) \leq \lambda_j^+ (S_{\rm K})
\end{align}
holds for all $j \in \N$. In particular,
\begin{align}\label{eq:FriedrichsKreinIneq}
 \lambda_j (S_{\rm F}) \leq \lambda_j^+ (S_{\rm K}) 
\end{align}
holds for all $j \in \N$. If $j \in \N$ is such that $\lambda_j (S_{\rm F})$ is not an eigenvalue of $S$, then the inequality~\eqref{eq:FriedrichsKreinIneq} is strict, that is, $\lambda_j (S_{\rm F}) < \lambda_j^+ (S_{\rm K})$.
\end{theorem}

\begin{proof}
Let us fix $j$ and choose a $j$-dimensional subspace $F$ of $\dom S$ such that
\begin{align*}
 \|S f\|^2 \leq \lambda_j^+ (S_{\rm K}) (S f, f) \qquad \text{for all}~f \in F.
\end{align*}
Then for any $f \in F$ and $g \in \ker A$ we have
\begin{align*}
 \big( A (f + g), f + g \big)^2 \leq \|A (f + g)\|^2 \|f + g\|^2,
\end{align*}
and hence
\begin{align}\label{eq:sum}
 \frac{\big(A (f + g), f + g \big)}{\|f + g\|^2} \leq \frac{\|A (f + g)\|^2}{\big( A (f + g), f + g \big)} = \frac{\|A f\|^2}{(A f, f)} = \frac{\|S f\|^2}{(S f, f)} = \lambda_j^+ (S_{\rm K}). 
\end{align}
Due to~\eqref{eq:mu}, $\ker A \cap \dom S = \{0\}$ and, thus $\dim (F + \ker A) = j + d$. Therefore~\eqref{eq:sum} together with the usual min-max principle for $A$ implies the assertion~\eqref{eq:indexShift}. Note that by the compactness of the embedding $\iota$, the spectrum of $S_{\rm F}$ is purely discrete, and thus~\eqref{eq:indexShift} implies~\eqref{eq:FriedrichsKreinIneq}. Finally, assume that $\lambda_j (S_{\rm F})$ is not an  eigenvalue of $S$, and let $g = 0$ in the estimate~\eqref{eq:sum}. Assuming $\lambda_j (S_{\rm F}) = \lambda_j^+ (S_{\rm K})$ for a contradiction, we get equality in~\eqref{eq:sum} for some nontrivial $f \in \dom S$, with $A= S_{\rm F}$. Then $f \in \ker (S_{\rm F} - \lambda_j (S_{\rm F})) \cap \dom S = \ker (S - \lambda_j (S_{\rm F}))$ follows, a contradiction.
\end{proof}

\begin{remark}
We wish to point out that compactness of the embedding of $\cH_S$ into~$\cH$ does not imply that all self-adjoint extensions of $S$ have a purely discrete spectrum. An example is the Krein--von Neumann extension of the Laplacian with both Dirichlet and Neumann boundary conditions on a bounded, sufficiently smooth domain in $\R^m$, $m \geq 2$, where $\ker S_{\rm K} = \ker S^*$ consists of all harmonic functions, and thus is infinite-dimensional, see, e.g.~\cite{AGMT10} for more details. 
\end{remark}

% We also want to remark that one cannot replace $\lambda_j^+ (S_{\rm K})$ in the above theorem by $\lambda_j (S_{\rm F})$. Counterexample e.g.\ interval with mixed BC. 

\begin{remark}
If the Krein--von Neumann extension of $S$ has purely discrete spectrum (in particular $d = \dim \ker S_{\rm K}$ is finite) we may choose $A = S_{\rm K}$ in Theorem~\ref{thm:inequalitiesAbstract}. As $\lambda_{j + d} (S_{\rm K}) = \lambda_j^+ (S_{\rm K})$, this shows that the inequality~\eqref{eq:indexShift} is not necessarily strict in general, not even if $S$ does not have any eigenvalues.
\end{remark}

Given two symmetric operators $S, \widetilde S$ in $\cH$ such that $S \subset \widetilde S$, we get the following interlacing properties of the positive eigenvalues of their respective Krein--von Neumann extensions. We will apply it several times in subsequent sections.

\begin{theorem}\label{thm:interlacingAbstract}
Let $S, \widetilde S$ be closed, densely defined, symmetric operators in $\cH$ with $S \subset \widetilde S$ such that \eqref{eq:mu} holds for $S$ replaced by $\widetilde S$. Moreover, assume that the embedding $\widetilde \iota : \cH_{\widetilde S} \to \cH$ is compact, and denote by $S_{\rm K}$ and $\widetilde S_{\rm K}$ the Krein--von Neumann extensions of $S$ and $\widetilde S$ respectively. Then $\sigma (S_{\rm K})$ and $\sigma (\widetilde S_{\rm K})$ are nonnegative and consist of isolated eigenvalues, and their positive eigenvalues have finite multiplicities. If we assume, in addition, that $\dom S$ is a subspace of $\dom \widetilde S$ of co-dimension $k$, then the positive eigenvalues of $S_{\rm K}$ and $\widetilde S_{\rm K}$ satisfy the interlacing inequalities
\begin{equation}\label{eq:interlacingAbstract}
 \lambda_j^+ (\widetilde S_{\rm K}) \leq \lambda_j^+ (S_{\rm K}) \leq \lambda_{j + k}^+ (\widetilde S_{\rm K}) \leq \lambda_{j + k}^+ (S_{\rm K})
\end{equation}
for all $j \in \N$.
\end{theorem}

\begin{proof}
Firstly, the assumption $S \subset \widetilde S$ implies $\cH_S \subset \cH_{\widetilde S}$ algebraically, together with
\begin{align*}
 \|f\|_S = \|f\|_{\widetilde S} \qquad \text{for all}~f \in \cH_S.
\end{align*}
Hence \eqref{eq:mu} follows also for $S$, and compactness of the embedding $\widetilde \iota$ implies compactness of the embedding $\iota : \cH_S \to \cH$. With the help of the latter, the discreteness statement on the spectra of $S_{\rm K}$ and $\widetilde S_{\rm K}$ follows from Theorem~\ref{thm:minMaxKrein}. 

Secondly, the first and third inequalities in~\eqref{eq:interlacingAbstract} follow directly from the inclusion $S \subset \widetilde S$, and the min-max principle in  Theorem~\ref{thm:minMaxKrein}. It remains to prove the middle inequality in~\eqref{eq:interlacingAbstract}. 
%We use a similar argument to that used in the proof of Theorem 4.3, \cite{BKKM19}.

Let $j \in \N$ and let $\widetilde F \subset \dom \widetilde S$ be any $(j + k)$-dimensional subspace of $\dom \widetilde S$ such that
\begin{align*}
 \max_{0 \neq f \in \widetilde F}\frac{\|\widetilde S f\|^2}{(\widetilde S f, f)} = \lambda_{j + k}^+ (\widetilde S_{\rm K}).
\end{align*}
As $\dom S$ is a subspace of $\dom \widetilde S$ of co-dimension $k$, the subspace $F := \widetilde F \cap \dom S$ of $\dom S$ satisfies $\dom F \geq j$, and we have
\begin{equation*}
 \lambda_j^+(S_{\rm K}) \leq \max_{0 \neq f \in F}\frac{\|S f\|^2}{(S f, f)} = \max_{0 \neq f \in F}\frac{\|\widetilde S f\|^2}{(\widetilde S f, f)} \leq \max_{0 \neq f \in \widetilde F}\frac{\|\widetilde S f\|^2}{(\widetilde S f, f)} = \lambda_{j + k}^+ (\widetilde S_{\rm K}),
\end{equation*}
which completes the proof.
\end{proof}

We conclude this subsection with a comment on additive perturbations of the Krein--von Neumann extension.

\begin{remark}\label{rem:perturbation}
Assume that $Q = Q^*$ is a bounded, nonnegative, everywhere defined operator in $\cH$. If $S$ is closed, symmetric, densely defined, and satisfies \eqref{eq:mu} then all these properties are also true for $S + Q$, and thus $S + Q$ has a Krein--von Neumann extension which we denote by $(S + Q)_{\rm K}$. It is remarkable that this operator does not coincide with $S_{\rm K} + Q$, the additively perturbed Krein--von Neumann extension of $S$. This is in contrast to the Friedrichs extension, for which $(S + Q)_{\rm F} = S_{\rm F} + Q$ holds. For instance, if $Q = I$ is the identity operator then $(S + I)_{\rm K}$ has a nontrivial kernel (coinciding with $\ker (S^* + I)$), whilst $S_{\rm K} + I$ is bounded from below by one. Nevertheless, $S_{\rm K} + Q$ is a self-adjoint, nonnegative extension of $S + Q$ and we know thus that
\begin{align*}
 \lambda_j \big( (S + Q)_{\rm K} \big) \leq \lambda_j (S_{\rm K} + Q)
\end{align*}
holds for all $j \in \N$. On the other hand, by our Theorem \ref{thm:inequalitiesAbstract} one has
\begin{align*}
 \lambda_{j + d} (S_{\rm K} + Q) \leq \lambda_j^+ \big( (S + Q)_{\rm K} \big)
\end{align*}
for all $j \in \N$, where $d := \dim \ker (S_{\rm K} + Q) \leq \dim \ker S_{\rm K}$. 
%This indicates that $S_{\rm K} + Q$ and $(S + Q)_{\rm K}$ may only coincide if $\dim \ker (S_{\rm K} + Q) = \dim \ker ((S + Q)_{\rm K})$, in which case they are at least isospectral.\marginpar{\tiny JR: Are we aware of any example where this actually happens, if not $Q = 0$?}
\end{remark}

\subsection{The Krein--von Neumann extension in the framework of boundary triples}

In this subsection, we review properties of the Krein--von Neumann extension in the framework of boundary triples. Our main focus is on a Krein--type formula that expresses the resolvent difference between the Krein--von Neumann extension and another self-adjoint extension of $S$ (as, e.g. the Friedrichs extension) in terms of abstract boundary operators. We assume Hypothesis \ref{hyp:abs} throughout. First we recall the definition of a boundary triple. 

\begin{definition}
Assume Hypothesis \ref{hyp:abs}. A triple $\{\cG, \Gamma_0, \Gamma_1\}$ consisting of a Hilbert space $(\cG, (\cdot, \cdot)_\cG)$ and two linear mappings $\Gamma_1, \Gamma_2 : \dom S^* \to \cG$ is called {\em boundary triple for $S^*$} if the following conditions are satisfied:
\begin{enumerate}
 \item the mapping $\{\Gamma_0, \Gamma_1\} : \dom S^* \to \cG \times \cG$ is surjective;
 \item the {\em abstract Green identity} 
 \begin{align*}
  (S^* f, g) - (f, S^* g) = (\Gamma_1 f, \Gamma_0 g)_\cG - (\Gamma_0 f, \Gamma_1 g)_\cG
 \end{align*}
 holds for all $f, g \in \dom S^*$.
\end{enumerate}
\end{definition}

We remark that boundary triples exist for any symmetric, densely defined operator $S$ with equal defect numbers, even without the requirement \eqref{eq:mu}. For a detailed review on boundary triples and literature references we refer the reader to, e.g. the recent monograph \cite{BHS20} or \cite[Chapter 14]{S12}. 

For any given boundary triple, we have $S^* \upharpoonright (\ker \Gamma_0 \cap \ker \Gamma_1) = S$, and two self-adjoint extensions of $S$ are especially distinguished, namely 
\begin{align}\label{eq:AB}
 A := S^* \upharpoonright \ker \Gamma_0 \qquad \text{and} \qquad B := S^* \upharpoonright \ker \Gamma_1.
\end{align}
A boundary triple comes with two operator-valued functions defined on the resolvent set $\rho (A)$ of $A$.

\begin{definition}
Let Hypothesis \ref{hyp:abs} be satisfied, and let $\{\cG, \Gamma_0, \Gamma_1\}$ be a boundary triple for $S^*$. The mappings 
\begin{align*}
 \gamma : \rho (A) \to \cB (\cG, \cH) \qquad \text{and} \qquad M : \rho (A) \to \cB (\cG)
\end{align*}
defined as
\begin{align*}
 \gamma (\lambda) \Gamma_0 f = f \qquad \text{and} \qquad M (\lambda) \Gamma_0 f = \Gamma_1 f
\end{align*}
for $f \in \ker (S^* - \lambda)$ are called {\em $\gamma$-field} and {\em Weyl function} respectively, associated with the boundary triple $\{\cG, \Gamma_0, \Gamma_1\}$.
\end{definition}

The well-definedness of $\gamma (\lambda)$ and $M (\lambda)$ is due to the direct sum decomposition
\begin{align*}
 \dom S^* = \dom A \dotplus \ker (S^* - \lambda), \qquad \lambda \in \rho (A).
\end{align*}
The operator $\gamma (\lambda)$ can be viewed as an abstract Poisson operator, and $M (\lambda)$ may be interpreted as an abstract Dirichlet-to-Neumann map. It is well-known that $\lambda \mapsto M (\lambda)$ is an operator-valued Herglotz--Nevanlinna--Pick function. In particular, $M (\lambda)$ is self-adjoint for $\lambda \in \rho (A) \cap \R$ (if such points exist, which is always the case if \eqref{eq:mu} is assumed).

Boundary triples can be used to characterise e.g. self-adjoint extensions of $S$ in terms of abstract boundary conditions of the form $\Gamma_1 f = \Theta \Gamma_0 f$ with a self-adjoint parameter $\Theta$ acting in $\cG$. In order to actually describe all self-adjoint extensions of $S$, one needs to allow not only self-adjoint operators $\Theta$ but so-called self-adjoint linear relations (or multi-valued linear operators), and we do not go into these details here. For us it is sufficient to know the following; see e.g.~\cite[Theorems~2.1.3, 2.6.1, and~2.6.2]{BHS20}.

\begin{proposition}\label{prop:KreinFormula}
Let Hypothesis \ref{hyp:abs} be satisfied, let $\{ \cG, \Gamma_0, \Gamma_1\}$ be a boundary triple for $S^*$, and let $\Theta$ be a self-adjoint operator in $\cG$. Then
\begin{align*}
 A_\Theta := S^* \upharpoonright \big\{ f \in \dom S^* : \Gamma_1 f = \Theta \Gamma_0 f \big\}
\end{align*}
is a self-adjoint extension of $S$. Moreover, if we denote by $\lambda \mapsto \gamma (\lambda)$ and $\lambda \mapsto M (\lambda)$ the corresponding $\gamma$-field and Weyl function respectively, and $A$ is defined in~\eqref{eq:AB}, then the following assertions hold.
\begin{enumerate}
 \item The point $\lambda \in \rho (A)$ is an eigenvalue of $A_\Theta$ if and only if 0 is an eigenvalue of $\Theta - M (\lambda)$.
 \item The point $\lambda \in \rho (A)$ belongs to $\rho (A_\Theta)$ if and only if $0 \in \rho (\Theta - M (\lambda))$.
 \item For all $\lambda \in \rho (A) \cap \rho (A_\Theta)$,
 \begin{align*}
  (A_\Theta - \lambda)^{-1} - (A - \lambda)^{-1} = \gamma (\lambda) \big( \Theta - M (\lambda) \big)^{-1} \gamma (\overline \lambda)^*
 \end{align*}
 holds.
\end{enumerate}
\end{proposition}

Characterisations analogous to item (i) in the previous theorem hold for other types of spectra too, such as the continuous or residual spectrum, but this is not of relevance for us in this work.

If the boundary triple is chosen such that $0 \in \rho (A)$, then the Krein--von Neumann extension of $S$ can be characterised in the following way; this is well-known, but for the convenience of the reader we repeat the short proof.

\begin{proposition}\label{prop:KreinBT}
Let Hypothesis \ref{hyp:abs} be satisfied, and let $\{ \cG, \Gamma_0, \Gamma_1\}$ be a boundary triple for $S^*$ such that $0 \in \rho (A)$. Moreover, let $\lambda \mapsto M (\lambda)$ denote the corresponding Weyl function. Then the Krein--von Neumann extension $S_{\rm K}$ of $S$ equals
\begin{align*}
 S_{\rm K} = S^* \upharpoonright \big\{ f \in \dom S^* : \Gamma_1 f = M (0) \Gamma_0 f \big\}.
\end{align*}
\end{proposition}

\begin{proof}
Since $M (0)$ is self-adjoint, the restriction of $S^*$ to all $f$ which satisfy $\Gamma_1 f = M (0) \Gamma_0 f$ is a self-adjoint extension of $S$ by Proposition~\ref{prop:KreinFormula}. Moreover, by definition, each $f \in \dom S_{\rm K}$ can be written uniquely as $f = f_S + f_*$ with $f_S \in \dom S$ and $f_* \in \ker S^*$, and therefore 
\begin{align*}
 \Gamma_1 f = \Gamma_1 f_* = M (0) \Gamma_0 f_* = M (0) \Gamma_0 f,
\end{align*}
where we have used $\dom S = \ker \Gamma_0 \cap \ker \Gamma_1$. This completes the proof.
\end{proof}

Now that we have this characterisation of the domain of $S_{\rm K}$ at hand, we may use the above Krein--type resolvent formula to express the difference to both the distinguished self-adjoint extensions $A$ and $B$ of $S$.

\begin{proposition}\label{prop:Krein}
Assume that Hypothesis \ref{hyp:abs} holds. Let $\{ \cG, \Gamma_0, \Gamma_1\}$ be a boundary triple for $S^*$, and let $\lambda \mapsto \gamma (\lambda)$ and $\lambda \mapsto M (\lambda)$ denote the corresponding $\gamma$-field and Weyl function respectively. Let $A$ and $B$ be given in~\eqref{eq:AB}, and assume that $0 \in \rho (A)$. Then the following identities hold.
\begin{enumerate}
 \item For all $\lambda \in \rho (A) \cap \rho (S_{\rm K})$,
 \begin{align}\label{eq:KreinA}
  (S_{\rm K} - \lambda)^{-1} - (A - \lambda)^{-1} = \gamma (\lambda) \big( M (0) - M (\lambda) \big)^{-1} \gamma (\overline \lambda)^*
 \end{align}
 holds.
 \item For all $\lambda \in \rho (B) \cap \rho (S_{\rm K}) \cap \rho (A)$, the operator $M (\lambda)$ is invertible with $M (\lambda)^{-1} \in \cB (\cG)$ and
 \begin{align}\label{eq:KreinB}
  (S_{\rm K} - \lambda)^{-1} - (B - \lambda)^{-1} = \gamma (\lambda) \big( M (0) - M (\lambda) \big)^{-1} M (0) M (\lambda)^{-1} \gamma (\overline \lambda)^*
 \end{align}
 holds.
\end{enumerate}
\end{proposition}

\begin{proof}
Assertion (i) follows directly from plugging the result of Proposition~\ref{prop:KreinBT} into the resolvent formula of Proposition~\ref{prop:KreinFormula}~(iii). On the other hand, the operator $B$ corresponds to the operator $A_\Theta$ with $\Theta = 0$, and hence
\begin{align*}
 (A - \lambda)^{-1} - (B - \lambda)^{-1} = \gamma (\lambda) M (\lambda)^{-1} \gamma (\overline{\lambda})^*
\end{align*}
for all $\lambda \in \rho (A) \cap \rho (B)$. For those $\lambda$ which additionally belong to $\rho (S_{\rm K})$, we combine the latter formula with assertion (i) of the present proposition to get
\begin{align*}
 (S_{\rm K} - \lambda)^{-1} - (B - \lambda)^{-1} = \gamma (\lambda) \Big[ \big( M (0) - M (\lambda) \big)^{-1} + M (\lambda)^{-1} \Big] \gamma (\overline \lambda)^*.
\end{align*}
From this, the assertion (ii) follows by an easy calculation left to the reader.
\end{proof}

The resolvent formulae in the previous proposition may be used to determine the rank of the resolvent differences as follows.

\begin{corollary}\label{cor:KreinB}
Assume that Hypothesis \ref{hyp:abs} is satisfied. Let $\{ \cG, \Gamma_0, \Gamma_1\}$ be a boundary triple for $S^*$ with Weyl function $\lambda \mapsto M (\lambda)$, and let $A, B$ be as defined in~\eqref{eq:AB}. Moreover, let $0 \in \rho (A)$. Then the following hold.
\begin{enumerate}
 \item For all $\lambda \in \rho (A) \cap \rho (S_{\rm K})$,
\begin{align*}
\dim \ran \left[ (S_{\rm K} - \lambda)^{-1} - (A - \lambda)^{-1} \right] = \dim \ker (S^*-\overline{\lambda}) = \dim \cG.
\end{align*}
\item For all $\lambda \in \rho (A) \cap \rho (B) \cap \rho (S_{\rm K})$,
\begin{align*}
 \dim \ran \left[ (S_{\rm K} - \lambda)^{-1} - (B - \lambda)^{-1} \right] = \dim \ran M (0).
\end{align*}
\end{enumerate}
\end{corollary}

\begin{proof}
This follows rather directly from formulas~\eqref{eq:KreinA} and~\eqref{eq:KreinB} in a way similar to the proof of~\cite[Theorem 2.8.3]{BHS20}. In fact, we use that $\gamma (\lambda) : \cG \to \ker (S^* - \lambda)$ is an isomorphism and that 
\begin{align*}
 \ker \gamma (\overline \lambda)^* = \left( \ran \gamma (\overline \lambda) \right)^\perp = \left( \ker (S^* - \overline{\lambda}) \right)^\perp.
\end{align*}
This implies that 
\begin{align*}
 \ran \left[ (S_{\rm K} - \lambda)^{-1} - (A - \lambda)^{-1} \right] = \ran \left[ (S_{\rm K} - \lambda)^{-1} - (A - \lambda)^{-1} \right] \upharpoonright \ker (S^* - \overline{\lambda}),
\end{align*}
for all $\lambda \in \rho (A) \cap \rho (S_{\rm K})$, with the same equation holding after replacing $A$ with $B$ for all $\lambda \in \rho (A) \cap \rho (B) \cap \rho (S_{\rm K})$. Finally, as 
\begin{align*}
 \gamma (\overline \lambda)^* \upharpoonright \ker (S^* - \overline{\lambda}) : \ker (S^* - \overline{\lambda}) \to \cG
\end{align*}
is an isomorphism and both $(M (0) - M (\lambda))^{-1}$ and $M (\lambda)^{-1}$ are isomorphisms of $\cG$, the desired result follows from~\eqref{eq:KreinA} and~\eqref{eq:KreinB}.
\end{proof}

\section{Perturbed Krein Laplacians on metric graphs}\label{sec:Krein}

In this section and all sections which follow, we assume that $\Gamma$ is a metric graph consisting of a vertex set $\cV$, an edge set $\cE$, and a length function $\ell : \cE \to (0, \infty)$ which assigns a length to each edge. Every edge $e \in \cE$ is identified with the interval $[0, \ell (e)]$, and this parametrisation gives rise to a natural metric on $\Gamma$. We will always assume that $\Gamma$ is finite, i.e. $V := |\cV|$ and $E := |\cE|$ are finite numbers, and we consider only connected graphs.

We view a function $f : \Gamma \to \C$ as a collection of functions $f_e : (0, \ell (e)) \to \C$, $e \in \cE$, and say, accordingly, that $f$ belongs to $L^2 (\Gamma)$ if $f_e \in L^2 (0, \ell (e))$ for each $e \in \cE$. In order to define Schr\"odinger operators on metric graphs we make use of the Sobolev spaces 
\begin{align*}
 \widetilde H^k (\Gamma) := \left\{ f \in L^2 (\Gamma) : f_e \in H^k (0, \ell (e))~\text{for each}~e \in \cE \right\},
\end{align*}
$k \in \N$. For functions in $\widetilde H^1 (\Gamma)$, we may talk about continuity at a vertex $v$, meaning that for any two edges $e, \hat e$ incident with $v$, the limit values (or traces) of $f_e$ and $f_{\hat e}$ at the endpoints of the edges corresponding to $v$ coincide. In this sense, we make use of the function space
\begin{align*}
 H^1 (\Gamma) := \left\{ f \in \widetilde H^1 (\Gamma) : f~\text{is continuous at each vertex} \right\}.
\end{align*}
Moreover, for $f \in \widetilde H^2 (\Gamma)$ and $v \in \cV$, we write
\begin{align*}
 \partial_\nu f (v) := \sum \partial f_e (v),
\end{align*}
where the sum is taken over all edges $e$ incident with $v$, and $\partial f_e (v)$ is the derivative of $f_e$ at the endpoint corresponding to $v$, taken in the direction pointing towards~$v$; if $e$ is a loop then both endpoints have to be taken into account.

We will consider Schr\"odinger operators on metric graphs with potentials that are, for simplicity, bounded. However, everything may be extended easily to form-bounded (i.e.\ $L^1$) potentials. We will always assume the following hypothesis.

\begin{hypothesis}\label{hyp}
On the finite, connected metric graph $\Gamma$, the potential $q : \Gamma \to \R$ is measurable and bounded, and $q (x) \geq 0$ holds for almost all $x \in \Gamma$.
\end{hypothesis}

Under Hypothesis~\ref{hyp}, we define the Schr\"odinger operator with potential $q$ subject to Dirichlet and Kirchhoff vertex conditions at all vertices,
\begin{align}\label{eq:SV}
\begin{split}
 (S f)_e & = - f_e'' + q_e f_e \qquad \text{on each edge}~e \in \cE, \\
 \dom S & = \widetilde H_0^2 (\Gamma) := \left\{ f \in \widetilde H^2 (\Gamma) \cap H^1 (\Gamma) : \partial_\nu f (v) = f (v) = 0~\text{for each}~v \in \cV \right\}.
\end{split}
\end{align}
It is easy to see that $S$ is a symmetric, nonnegative, densely defined operator in the Hilbert space $L^2 (\Gamma)$. Since $\oplus_{e \in \cE} C_0^\infty (0, \ell (e)) \subset \dom S$, the Friedrichs extension of $S$ is the operator $- \Delta_{{\rm D}, \Gamma, q}$, called the \emph{perturbed Dirichlet Laplacian}, given by
\begin{align*}
 (- \Delta_{{\rm D},\Gamma, q} f)_e & = - f_e'' + q_e f_e \qquad \text{on each edge}~e \in \cE, \\
 \dom (- \Delta_{{\rm D},\Gamma, q}) & = \left\{ f \in \widetilde H^2 (\Gamma) \cap H^1 (\Gamma) : f (v) = 0~\text{for each}~v \in \cV \right\};
\end{align*}
if $q = 0$ identically, we just write $- \Delta_{{\rm D}, \Gamma}$ and call it the \emph{Dirichlet Laplacian}. The operator $- \Delta_{{\rm D}, \Gamma, q}$ has a purely discrete spectrum. In the case $q = 0$ identically, the latter is given by
\begin{align}\label{eq:VDirichleteigenvalues}
 \sigma (- \Delta_{{\rm D},\Gamma}) = \left\{ \lambda = \frac{k^2 \pi^2}{\ell (e)^2} : e \in \cE, k = 1, 2, \dots \right\},
\end{align}
where the multiplicity of an eigenvalue $\lambda$ coincides with the number of values $k$ and edges $e$ for which $\lambda = \frac{k^2 \pi^2}{\ell (e)^2}$. In particular, 
\begin{align*}
 \min \sigma (- \Delta_{{\rm D},\Gamma, q}) \geq \min \sigma (- \Delta_{{\rm D},\Gamma}) = \frac{\pi^2}{(\max_{e \in \cE} \ell (e))^2} =: \mu > 0,
\end{align*}
where we have used the assumption that $q$ is nonnegative, and the inclusion $S \subset - \Delta_{{\rm D},\Gamma, q}$ implies 
\begin{align*}
 (S f, f) \geq \mu \|f\|^2, \quad f \in \dom S,
\end{align*}
where $(\cdot, \cdot)$ and $\| \cdot \|$ denote the inner product and norm respectively in $L^2 (\Gamma)$. By an easy integration by parts, the adjoint of $S$ is given by
\begin{align*}
 (S^* f)_e & = - f_e'' + q_e f_e \qquad \text{on each edge}~e \in \cE, \\
 \dom S^* & = \widetilde H^2 (\Gamma) \cap H^1 (\Gamma).
\end{align*}
The two self-adjoint extensions of $S$ in focus here will be the Krein--von Neumann extension of $S$ and the Schr\"odinger operator with \emph{standard vertex conditions} (also called natural or continuity-Kirchhoff conditions), namely continuity and the Kirchhoff condition $\partial_\nu f=0$, at all vertices.

\begin{definition}
We assume that Hypothesis~\ref{hyp} is satisfied.
\begin{enumerate}
 \item The {\em perturbed Krein Laplacian} on $\Gamma$ is the Krein--von Neumann extension
\begin{align*}
 - \Delta_{{\rm K}, \Gamma, q} := S_{\rm K}
\end{align*}
of $S$. 
 \item The \emph{perturbed standard Laplacian} on $\Gamma$ is the operator given by
\begin{align*}
 (- \Delta_{{\rm st},\Gamma,q} f)_e & = - f_e'' + q_e f_e \qquad \text{on each edge}~e \in \cE, \\
 \dom (- \Delta_{{\rm st},\Gamma,q}) & = \left\{ f \in \widetilde H^2 (\Gamma) \cap H^1 (\Gamma) : \partial_\nu f (v) = 0~\text{for each}~v \in \cV \right\}.
\end{align*}
\end{enumerate}
In the case that the potential $q$ is identically zero, we write $ - \Delta_{{\rm K}, \Gamma} := - \Delta_{{\rm K}, \Gamma,0}$ and $- \Delta_{{\rm st},\Gamma} := - \Delta_{{\rm st}, \Gamma, 0}$ and call these operators {\em Krein Laplacian} and \emph{standard Laplacian}, respectively.
\end{definition}

We point out that, in general, $ - \Delta_{{\rm K}, \Gamma, q} \ne - \Delta_{{\rm K}, \Gamma} + q$ (where we interpret the latter as an additive perturbation of the Krein Laplacian); see the discussion in Remark~\ref{rem:perturbation}. On the other hand, it holds that $- \Delta_{{\rm st},\Gamma,q}=- \Delta_{{\rm st},\Gamma} +q$, by definition.

In what follows, it will be useful to embed the study of $- \Delta_{{\rm K}, \Gamma, q}$ in the framework of boundary triples. The following proposition can be found in~\cite[Lemma~2.14 and Theorem~2.16]{EK14}; see also~\cite[Proposition~10.1]{BLLR18}. For the statement on the weighted discrete Laplacian, see e.g. Step 2 in the proof of \cite[Proposition 3.1]{GR20}.

\begin{proposition}\label{prop:graphBT}
Assume that Hypothesis \ref{hyp} is satisfied, and let $S$ be defined in~\eqref{eq:SV}. For $f \in \dom S^* = \widetilde H^2 (\Gamma) \cap H^1 (\Gamma)$, define
\begin{align*}
 \Gamma_0 f = \begin{pmatrix} f (v_1) \\ \vdots \\ f (v_V) \end{pmatrix} \quad \text{and} \quad \Gamma_1 f = \begin{pmatrix} - \partial_\nu f (v_1) \\ \vdots \\ - \partial_\nu f (v_V) \end{pmatrix},
\end{align*}
where $v_1, \dots, v_V$ is an enumeration of the vertices of $\Gamma$. Then $S$ is a closed operator and $\{ \C^V, \Gamma_0, \Gamma_1\}$ is a boundary triple for $S^*$; in particular, $S$ has defect numbers
\begin{align}\label{eq:defectV}
 n_- = n_+ = V.
\end{align}
The corresponding extensions $A$ and $B$ of $S$ defined in~\eqref{eq:AB} are given by
\begin{align*}
 A = - \Delta_{{\rm D}, \Gamma, q} \quad \text{and} \quad B = - \Delta_{{\rm st}, \Gamma, q};
\end{align*}
in particular, $0 \in \rho (A)$. The value of the corresponding Weyl function at $\lambda = 0$ is $M (0) = - \Lambda_q$, where $\Lambda_q$ is the {\em Dirichlet-to-Neumann matrix} defined via the relation
\begin{align}\label{eq:DN}
 \begin{pmatrix} \partial_\nu f_* (v_1) \\ \vdots \\ \partial_\nu f_* (v_V) \end{pmatrix} = \Lambda_q \begin{pmatrix} f_* (v_1) \\ \vdots \\ f_* (v_V) \end{pmatrix},
\end{align}
where $f_* \in \ker S^*$ is arbitrary. In the potential-free case, $q = 0$ identically, the value of the corresponding Weyl function is $M (0) = - \Lambda_0 = - L$, where $L$ is the weighted discrete Laplacian $L$ defined as
\begin{align}\label{eq:L}
 L_{i, j} = \begin{cases} - \sum_{e~\text{connects $v_i$ and $v_j$}} \frac{1}{L (e)} & \text{if}~v_i~\text{and}~v_j~\text{are adjacent}, i \neq j,\\ 0 & \text{if}~v_i, v_j~\text{are not adjacent}, \\ \sum_{e \in \cE (v_i),~e~\text{no loop}} \frac{1}{L (e)} & \text{if}~i = j.  \end{cases}
\end{align}
\end{proposition}

This proposition allows us to describe the domain of $- \Delta_{{\rm K}, \Gamma, q}$ in terms of its vertex conditions and to obtain some properties of the perturbed Krein Laplacian right away. The next proposition follows immediately from Proposition~\ref{prop:graphBT} together with Proposition~\ref{prop:KreinBT}. Furthermore, from \eqref{eq:defectV} we obtain the multiplicity of the zero eigenvalue of $- \Delta_{{\rm K}, \Gamma, q}$.

\begin{proposition}\label{prop:KreinV}
Under Hypothesis \ref{hyp} the perturbed Krein Laplacian acts as
\begin{align*}
 \big( - \Delta_{{\rm K}, \Gamma, q} f \big)_e = - f_e'' + q_e f_e \qquad \text{on each edge}~e \in \cE,
\end{align*}
and its domain consists of all $f \in \widetilde H^2 (\Gamma) \cap H^1 (\Gamma)$ such that
\begin{align*}
 \begin{pmatrix} \partial_\nu f (v_1) \\ \vdots \\ \partial_\nu f (v_V) \end{pmatrix} = \Lambda_q \begin{pmatrix} f (v_1) \\ \vdots \\ f (v_V) \end{pmatrix},
\end{align*}
where $\Lambda_q$ is the Dirichlet-to-Neumann matrix defined in~\eqref{eq:DN}. Moreover,
\begin{align}\label{eq:Vkernel}
 \dim \ker \big( - \Delta_{{\rm K}, \Gamma, q} \big) = \dim\ker S^* = V.
\end{align}
In the potential-free case $q = 0$ identically, the domain of $- \Delta_{{\rm K},\Gamma}$ consists of all $f \in \widetilde H^2 (\Gamma) \cap H^1 (\Gamma)$ which satisfy the vertex conditions
\begin{align}\label{eq:KreinConditionsVLaplace}
 \begin{pmatrix} \partial_\nu f (v_1) \\ \vdots \\ \partial_\nu f (v_V) \end{pmatrix} = L \begin{pmatrix} f (v_1) \\ \vdots \\ f (v_V) \end{pmatrix},
\end{align}
where $L$ is the weighted discrete Laplacian in \eqref{eq:L}. 
\end{proposition} 

\begin{remark}
The vertex conditions of $- \Delta_{{\rm K}, \Gamma, q}$ are nonlocal, i.e.\ they couple values of the  function and its derivatives at different vertices. In the potential-free case it actually follows from \eqref{eq:KreinConditionsVLaplace} that the vertex conditions of the Krein Laplacian couple each vertex with all of its neighbours.
\end{remark}

We demonstrate the application of Proposition \ref{prop:KreinV} to the calculation of the Krein Laplacian for the interval --- a very standard example --- and a flower graph.

\begin{example}\label{ex:interval}
Let $\Gamma= [0, \ell]$ be an interval, i.e.\ a graph consisting of two vertices and one edge between them. On this graph, the weighted discrete Laplacian $L$ defined in~\eqref{eq:L} equals
\begin{align*}
 L = \frac{1}{\ell} \begin{pmatrix} 1 & -1 \\ -1 & 1 \end{pmatrix}
\end{align*}
and the vertex condition for the Krein Laplacian $-\Delta_{{\rm K},\Gamma}$ as described in Proposition~\ref{prop:KreinV} can be rewritten
\begin{align*}
 f' (0) = f' (\ell), \qquad f (\ell) = f (0) + \ell f' (0).
\end{align*}
\end{example}

Our second example shows that the Krein Laplacian and the standard Laplacian may coincide in some cases; cf.\ Corollary~\ref{cor:flower} below.

\begin{example}\label{ex:flower}
Let $\Gamma$ be a flower graph, i.e.\ a graph with one vertex and $E$ loops attached to it; special cases are loops ($E = 1$) and figure-8 graphs ($E = 2$); cf.\ Figure \ref{fig:flower}.
\begin{figure}[h]
\begin{center}
% flower
\begin{minipage}[c]{4cm}
\begin{tikzpicture}[scale=0.6]
\begin{polaraxis}[grid=none, axis lines=none]
\addplot[very thick,mark=none,domain=0:360,samples=300] {abs(cos(5*x/2))};
\draw[fill] (0,0) circle (2.5pt);
\end{polaraxis}
\end{tikzpicture}
\end{minipage}\hspace{-1cm}
% loop
\begin{minipage}[c]{4cm}
\begin{tikzpicture}[scale=0.6]
\begin{polaraxis}[grid=none, axis lines=none]
\addplot[very thick,mark=none,domain=-90:90,samples=300] {cos(x)^2};
\draw[fill] (0,0) circle (2.5pt);
% \draw[fill] (-700,0) circle (2.5pt);
% \draw[very thick] (0,0) edge (-700,0);
\end{polaraxis}
\end{tikzpicture}
\end{minipage}\hspace{0.8cm}
% figure-8
\begin{minipage}[c]{4cm}
\begin{tikzpicture}[scale=0.6]
\begin{polaraxis}[grid=none, axis lines=none]
\addplot[very thick,mark=none,domain=0:360,samples=300] {cos(x)^2};
\draw[fill] (0,0) circle (2.5pt);
\end{polaraxis}
\end{tikzpicture}
\end{minipage}
\end{center}
\caption{A ``generic'' flower graph and two special cases, the loop and the figure-8 graph.}
\label{fig:flower}
\end{figure}
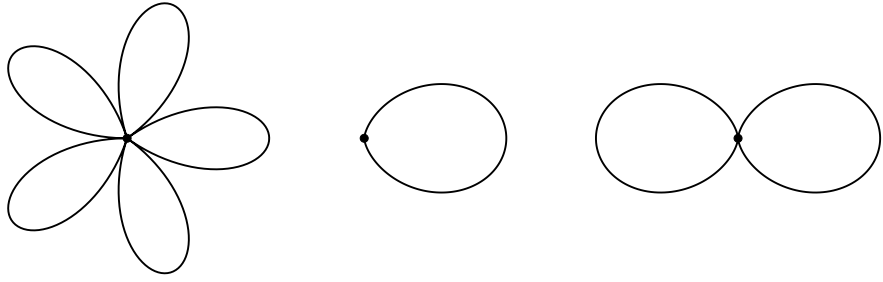

Then any function $f_*$ which is harmonic on every edge and belongs to~$H^1 (\Gamma)$ is necessarily constant on all of $\Gamma$. Thus each $f \in \dom (- \Delta_{{\rm K},\Gamma})$ satisfies $f = f_S + c$ with $f_S \in \dom S$ and $c$ constant; in particular, $f \in \dom (- \Delta_{{\rm st},\Gamma})$, the domain of the standard Laplacian on $\Gamma$. As both $- \Delta_{{\rm K},\Gamma}$ and $- \Delta_{{\rm st},\Gamma}$ are self-adjoint operators, they coincide, $- \Delta_{{\rm K},\Gamma} = - \Delta_{{\rm st},\Gamma}$, on any flower graph $\Gamma$.
\end{example}

Next we compare the perturbed Krein Laplacian with the perturbed Dirichlet Laplacian and the perturbed standard Laplacian. We apply Proposition~\ref{prop:KreinV} and Corollary~\ref{cor:KreinB} to the boundary triple in Proposition~\ref{prop:graphBT} and get the following result.

\begin{theorem}\label{thm:resDiffGraph}
Assume that Hypothesis \ref{hyp} is satisfied. Let $\lambda \mapsto \gamma (\lambda)$ and $\lambda \mapsto M (\lambda)$ be the $\gamma$-field and Weyl function respectively corresponding to the boundary triple in Proposition~\ref{prop:graphBT}.
\begin{enumerate}
 \item For $\lambda \in \rho (- \Delta_{{\rm K}, \Gamma, q}) \cap \rho (- \Delta_{{\rm D}, \Gamma, q})$, the formula
\begin{align*}
  \big(- \Delta_{{\rm K}, \Gamma, q} - \lambda \big)^{-1} - \big(- \Delta_{{\rm D}, \Gamma, q} - \lambda \big)^{-1} = -\gamma (\lambda) \big( \Lambda_q + M (\lambda) \big)^{-1} \gamma (\overline \lambda)^*
 \end{align*}
holds. In particular, 
\begin{align*}
 \dim \ran \Big[ \big(- \Delta_{{\rm K}, \Gamma, q} - \lambda \big)^{-1} & - \big(- \Delta_{{\rm D}, \Gamma, q} - \lambda \big)^{-1} \Big]  = V.
\end{align*}

\item For $\lambda \in \rho (- \Delta_{{\rm K}, \Gamma, q}) \cap \rho(- \Delta_{{\rm st}, \Gamma, q}) \cap \rho (- \Delta_{{\rm D}, \Gamma, q})$, the formula
\begin{align*}
  \big(- \Delta_{{\rm K}, \Gamma, q} - \lambda \big)^{-1} - \big(- \Delta_{{\rm st}, \Gamma, q} - \lambda \big)^{-1} = \gamma (\lambda) \big( \Lambda_q + M (\lambda) \big)^{-1} \Lambda_q M (\lambda)^{-1} \gamma (\overline \lambda)^*
 \end{align*}
holds. In particular, 
\begin{align*}
 \dim \ran \Big[ \big(- \Delta_{{\rm K}, \Gamma, q} - \lambda \big)^{-1} & - \big(- \Delta_{{\rm st}, \Gamma, q} - \lambda \big)^{-1} \Big] \\
 & = \dim \ran \Lambda_q = \begin{cases} V - 1 & \text{if}~q = 0~\text{identically}, \\ V, & \text{else}. \end{cases}
\end{align*}
\end{enumerate}
\end{theorem}

\begin{proof}
The only assertion to prove is that $\ran \Lambda_q$ has the dimension claimed in the theorem. For the potential-free case, where $\Lambda_q = L$, the weighted discrete Laplacian, it is well-known that the kernel is one-dimensional (consisting of the constant vectors), and hence its range has dimension $V - 1$. Now let $q \geq 0$ be a nontrivial function, and let $\phi \in \ker \Lambda_q$. Then by definition, there exists a unique $f \in \ker S^*$ such that $\Gamma_0 f = \phi$ and $\Gamma_1 f = 0$, i.e.\ $f \in \ker S^* \cap \dom (- \Delta_{{\rm st}, \Gamma, q})$. In other words, $f \in \ker (- \Delta_{{\rm st}, \Gamma, q})$. But then, by standard variational principles,
\begin{align*}
 0 = \int_\Gamma |f'|^2 \dd x + \int_\Gamma q |f|^2 \dd x.
\end{align*}
Since both terms on the right-hand side are nonnegative, they are zero separately. From $\int_\Gamma |f'|^2 \dd x = 0$, it follows that $f$ is constant on each edge and, by continuity, constant on $\Gamma$. But then $\int_\Gamma q |f|^2 \dd x = 0$ yields $f = 0$ identically, as $q$ is nontrivial. Finally, $\phi = \Gamma_0 f = 0$, so that $\ker \Lambda_q = \{0\}$. Consequently, $\dim \ran \Lambda_q = V$, which yields the desired result.
\end{proof}

Now the observation of Example~\ref{ex:flower} can be sharpened in the following way. Since flower graphs are the only graphs with $V = 1$, this is an immediate consequence of Theorem~\ref{thm:resDiffGraph}.

\begin{corollary}\label{cor:flower}
Under Hypothesis \ref{hyp}, the following statements are equivalent.
\begin{enumerate}
 \item The perturbed Krein Laplacian and the perturbed standard Laplacian coincide, i.e.\ $- \Delta_{{\rm K}, \Gamma, q} = - \Delta_{{\rm st},\Gamma, q}$;
 \item $\Gamma$ is a flower graph and $q = 0$ identically.
\end{enumerate}
\end{corollary}

% We remind the reader of the following fact; see, e.g.,~\cite[Theorem~3, Chapter~9, Section~6]{BS87}.
% 
% 
% \begin{proposition}\label{prop:resolventdifference}
% Let $A$ and $B$ be self-adjoint operators in a Hilbert space $\cH$ and suppose that there exists some $\lambda\in\rho(A)\cap\rho(B)$ for which
% \begin{equation*}
% r:=\dim \ran \Big[ \big(A - \lambda \big)^{-1} - \big(B - \lambda \big)^{-1} \Big]
% \end{equation*}
% is finite. Then for any bounded interval $J\subset\mathbb{R}$, the spectrum of $A$ in $J$ is discrete if and only if the spectrum of $B$ in $J$ is discrete. In this case, the numbers of eigenvalues of $A$ and $B$ in $J$, counting multiplicities, satisfy the inequalities
% \begin{equation}\label{eq:discretespectrum}
% |\sigma(B)\cap J| - r \leq |\sigma(A)\cap J| \leq |\sigma(B)\cap J| + r.
% \end{equation}
% \end{proposition}
% Further, if the spectrum of $A$ (or $B$) satisfying the hypothesis of Proposition~\ref{prop:resolventdifference} is discrete and lower-semibounded, 
% then \eqref{eq:discretespectrum} holds for any upper-semibounded interval $J$. 
% If $A$ has discrete lower-semibounded spectrum, we denote by $\mathcal{N}(\cdot;A)$ its eigenvalue-counting function $A$, 
% i.e. $\mathcal{N}(\lambda;A):=|\sigma(A)\cap(-\infty,\lambda]|$ is number of eigenvalues of $A$ (counting multiplicities) less than or equal to $\lambda$. 

Theorem~\ref{thm:resDiffGraph} allows us to deduce eigenvalue asymptotics for the perturbed Krein Laplacian: for $\lambda \geq 0$ denote by 
\begin{align*}
 \mathcal{N} \big(\lambda; - \Delta_{\bullet, \Gamma, q} \big) := \big| \sigma \big( - \Delta_{\bullet, \Gamma, q} \big) \cap (-\infty,\lambda] \big|, \qquad \bullet = {\rm K, D, st},
\end{align*}
 the \emph{eigenvalue-counting function} of the respective operator, that is, the number of eigenvalues up to $\lambda$. Under Hypothesis~\ref{hyp}, it follows immediately from Theorem~\ref{thm:resDiffGraph} and the minimality property of the Krein--von Neumann extension that
\begin{align}
 \mathcal{N} \big(\lambda;- \Delta_{{\rm D}, \Gamma,q} \big) & \leq \mathcal{N}\big(\lambda;- \Delta_{{\rm K}, \Gamma,q}\big) \leq \mathcal{N} \big(\lambda;- \Delta_{{\rm D}, \Gamma,q} \big)+V, \label{eq:countDV}\\
 \mathcal{N} \big(\lambda;- \Delta_{{\rm st}, \Gamma,q} \big) & \leq \mathcal{N}\big(\lambda;- \Delta_{{\rm K}, \Gamma,q} \big) \leq \mathcal{N} \big(\lambda;- \Delta_{{\rm st}, \Gamma,q} \big)+V.\nonumber
\end{align}
In the case that $q$ is identically zero on $\Gamma$, the latter inequality may be strengthened,
\begin{align}\label{eq:countstVLaplace}
 \mathcal{N} \big(\lambda;- \Delta_{{\rm st}, \Gamma} \big) & \leq \mathcal{N}\big(\lambda;- \Delta_{{\rm K}, \Gamma} \big) \leq \mathcal{N} \big(\lambda;- \Delta_{{\rm st}, \Gamma} \big) + V - 1.
\end{align}
Morover, one can use the inequalities for $- \Delta_{{\rm D}, \Gamma}$ in this case to deduce the following.

\begin{corollary}\label{cor:VWeyl}
In the case of zero potential $q\equiv0$, for any $\lambda\geq0$,
\begin{equation*} 
 \frac{\ell(\Gamma)}{\pi}\sqrt{\lambda}-E \leq \mathcal{N} \big(\lambda;- \Delta_{{\rm K}, \Gamma} \big) \leq \frac{\ell(\Gamma)}{\pi}\sqrt{\lambda}+V.
\end{equation*}
\end{corollary}

\begin{proof}
It is a straightforward exercise to show that
\begin{equation*}
\mathcal{N}(\lambda;- \Delta_{{\rm D}, \Gamma}) = \sum_{e\in \cE} \left\lfloor \frac{\ell(e)}{\pi}\sqrt{\lambda}\right\rfloor
\end{equation*}
follows from \eqref{eq:VDirichleteigenvalues}. In particular, this implies
\begin{equation*}
 \frac{\ell(\Gamma)}{\pi}\sqrt{\lambda}-E \leq \mathcal{N}(\lambda;- \Delta_{{\rm D}, \Gamma}) \leq \frac{\ell(\Gamma)}{\pi}\sqrt{\lambda},
\end{equation*}
and then inserting this into \eqref{eq:countDV} yields the desired result.
\end{proof}

One can immediately deduce from Corollary~\ref{cor:VWeyl} that the eigenvalues for $- \Delta_{{\rm K}, \Gamma}$ possess the Weyl asymptotics
\begin{equation*}
\lambda_j(- \Delta_{{\rm K}, \Gamma}) \sim \left(\frac{j \pi}{\ell(\Gamma)}\right)^2
\end{equation*}
as $j\to\infty$. However, we remark that in fact any self-adjoint extension of the operator $S$ given by \eqref{eq:SV} possesses these same asymptotics.

In the following, we are going to state some eigenvalue inequalities for the perturbed Krein Laplacian. It follows directly from \eqref{eq:Vkernel} that
\begin{align*}
\lambda_{j + V} (- \Delta_{{\rm K}, \Gamma, q})=\lambda_j^+ (- \Delta_{{\rm K}, \Gamma, q})
\end{align*}
holds for all $j\in\mathbb{N}$. To investigate properties of the positive eigenvalues of $- \Delta_{{\rm K},\Gamma,q}$, we first formulate the abstract variational principle in Theorem~\ref{thm:minMaxKrein} in our specific situation.

\begin{theorem}\label{thm:minMaxKreinV}
If Hypothesis \ref{hyp} is satisfied, then the spectrum of $- \Delta_{{\rm K}, \Gamma, q}$ is purely discrete, and the positive eigenvalues 
\begin{align*}
 \lambda_1^+ \big( - \Delta_{{\rm K}, \Gamma, q} \big) \leq \lambda_2^+ \big( - \Delta_{{\rm K}, \Gamma, q} \big) \leq \dots
\end{align*}
of $- \Delta_{{\rm K}, \Gamma, q}$, counted with multiplicities, satisfy
\begin{align*}
 \lambda_j^+ \big( - \Delta_{{\rm K}, \Gamma, q} \big) = \min_{\substack{F \subset \widetilde H_0^2 (\Gamma) \\ \dim F = j}} \max_{\substack{f \in F \\ f \neq 0}} \frac{\int_\Gamma \left|- f'' + q f\right|^2 \dd x}{\int_\Gamma |f'|^2 \dd x + \int_\Gamma q |f|^2 \dd x}
\end{align*}
for all $j \in \N$. In particular, in the potential-free case $q = 0$ identically,
\begin{align*}
 \lambda_j^+ (- \Delta_{{\rm K},\Gamma}) = \min_{\substack{F \subset \widetilde H_0^2 (\Gamma) \\ \dim F = j}} \max_{\substack{f \in F \\ f \neq 0}} \frac{\int_\Gamma |f''|^2 \dd x}{\int_\Gamma |f'|^2 \dd x}
\end{align*}
holds for all $j \in \N$.
\end{theorem}

The following eigenvalue inequalities and equalities are direct consequences of Theorem~\ref{thm:inequalitiesAbstract} and~\eqref{eq:Vkernel}.

\begin{theorem}\label{thm:inequalitiesV}
Let Hypothesis \ref{hyp} be satisfied, let $- \Delta_{\Gamma, q}$ be any self-adjoint extension of the operator $S$ in~\eqref{eq:SV}, and let $d := \dim \ker (- \Delta_{\Gamma, q})$. Then
\begin{align*}
 \lambda_{j + d} (- \Delta_{\Gamma, q}) \leq \lambda_j^+ (- \Delta_{{\rm K}, \Gamma, q}) = \lambda_{j + V} (- \Delta_{{\rm K}, \Gamma, q})
\end{align*}
holds for all $j \in \N$. In particular, 
\begin{align}\label{eq:DirichletKreinV}
 \lambda_j (- \Delta_{{\rm D}, \Gamma, q}) \leq \lambda_j^+ (- \Delta_{{\rm K}, \Gamma, q}) = \lambda_{j + V} (- \Delta_{{\rm K}, \Gamma, q})
\end{align}
holds for all $j \in \N$, and in the potential-free case we have
\begin{align}\label{eq:standardKreinV}
 \lambda_{j + 1} (- \Delta_{{\rm st},\Gamma}) \leq \lambda_j^+ (- \Delta_{{\rm K},\Gamma}) = \lambda_{j + V} (- \Delta_{{\rm K},\Gamma})
\end{align}
for all $j \in \N$.
\end{theorem}

\begin{remark}
The inequalities \eqref{eq:DirichletKreinV} and~\eqref{eq:standardKreinV} can alternatively be deduced from Theorem~\ref{thm:resDiffGraph}; cf.\ \eqref{eq:countDV} and \eqref{eq:countstVLaplace}.
\end{remark}

\begin{remark}
If the edge lengths in $\Gamma$ are rationally independent and $q \equiv 0$, then the inequality~\eqref{eq:DirichletKreinV} is strict for all $j \in \N$, as in this case it can be seen easily that $S$ does not possess any eigenvalues.
\end{remark}

\begin{remark}
If $\Gamma$ is a tree graph, then it is known that for the Laplacian,
\begin{align}\label{eq:Friedlander}
 \lambda_{j + 1} (- \Delta_{{\rm st},\Gamma}) \leq \lambda_j (- \Delta_{{\rm D},\Gamma})
\end{align}
holds for all $j \in \N$; see e.g. \cite[Theorem 4.1]{R17}. One may combine this with \eqref{eq:DirichletKreinV} to obtain \eqref{eq:standardKreinV} in an alternative way. However, it is worth pointing out that \eqref{eq:Friedlander} does not hold in general for graphs with cycles (see the discussion in \cite[Section 5]{KR20}), but in this case \eqref{eq:standardKreinV} is still true.
\end{remark}

% 
% 
% We would like to point out a monotonicity property of the eigenvalues of the perturbed Krein Laplacian with respect to the potential. This is not entirely obvious since adding a potential, as we pointed out before, is not just an additive perturbation of the Krein Laplacian. Nevertheless the following theorem holds.
% 
% \begin{theorem}
% Assume that Hypothesis \ref{hyp} is satisfied and that $\widetilde q : \Gamma \to \R$ is a further potential with $q (x) \leq \widetilde q (x)$ for almost all $x \in \Gamma$. Then 
% \begin{align}\label{eq:potentialPositive}
%  \lambda_j^+ \big( - \Delta_{{\rm K}, \Gamma, q} \big) \leq \lambda_j^+ \big( - \Delta_{{\rm K}, \Gamma, \widetilde q} \big)
% \end{align}
% and
% \begin{align}\label{eq:potentialAll}
%  \lambda_j \big( - \Delta_{{\rm K}, \Gamma, q} \big) \leq \lambda_j \big( - \Delta_{{\rm K}, \Gamma, \widetilde q} \big)
% \end{align}
% hold for all $j \in \N$.
% \end{theorem}
% 
% \begin{proof}
% Note that the operator $- \Delta_{{\rm K}, \Gamma, q} + \widetilde q - q$ is a self-adjoint extension of the operator $S$ in \eqref{eq:SV} with potential $\widetilde q$, and 
% \begin{align*}
%  d := \dim \ker \big( 
% \end{align*}
% 
% \end{proof}
% 
% 
% 

\section{Spectral implications of graph surgery operations}\label{sec:surgery}

Next, we investigate the effect of graph surgery operations on the eigenvalues of the perturbed Krein Laplacian $- \Delta_{{\rm K}, \Gamma, q}$. 
Graph surgery refers to the process of transforming the operator by making topological and metric changes to the graph, 
such as gluing vertices together or adding edges, forming a new graph $\widetilde\Gamma$.
One associates a potential $\widetilde q$ to the new graph $\widetilde \Gamma$ which will be determined by the type of surgery carried out.
Given a surgery operation $- \Delta_{{\rm K}, \Gamma, q}\mapsto- \Delta_{{\rm K}, \widetilde\Gamma, \widetilde q}$,
only the operators  $- \Delta_{{\rm K}, \Gamma, q}$ and $- \Delta_{{\rm K}, \widetilde\Gamma, \widetilde q}$ will be of significance to us, 
and thus we use the following simplified notation for their eigenvalues throughout this section:
\begin{align*}
\lambda_j^+:= & \lambda_j^+(- \Delta_{{\rm K}, \Gamma, q}), & \widetilde \lambda_j^+:= & \lambda_j^+(- \Delta_{{\rm K}, \widetilde \Gamma, \widetilde q}), \\
\lambda_j:= & \lambda_j(- \Delta_{{\rm K}, \Gamma, q}), & \widetilde \lambda_j:= & \lambda_j(- \Delta_{{\rm K}, \widetilde \Gamma, \widetilde q}). 
\end{align*}
In what follows, we always assume Hypothesis \ref{hyp}; the new potential $\widetilde q$ will satisfy the analogue of Hypothesis \ref{hyp} conditions for $\widetilde\Gamma$ by construction.

We begin with transformations which only affect the vertex conditions of the operator, or add new vertices. For such operations, the potential $\widetilde q\equiv q$ is unchanged (except possibly on a set of measure zero).

\begin{definition}
Let $\widetilde\Gamma$ be the graph formed from $\Gamma$ by identifying a number of its vertices, say $v_1,\dots,v_{k+1}$, to form a new vertex $v_0$. The total number of vertices is thereby reduced by $k$, and the potential $q$ associated with $\Gamma$ remains well-defined on $\widetilde\Gamma$. The transformation $- \Delta_{{\rm K}, \Gamma, q}\mapsto- \Delta_{{\rm K}, \widetilde \Gamma, q}$ is called \emph{gluing vertices}, and the inverse operation is referred to as \emph{cutting through vertices}; cf.\ Figure \ref{fig:gluingV}.
\end{definition}

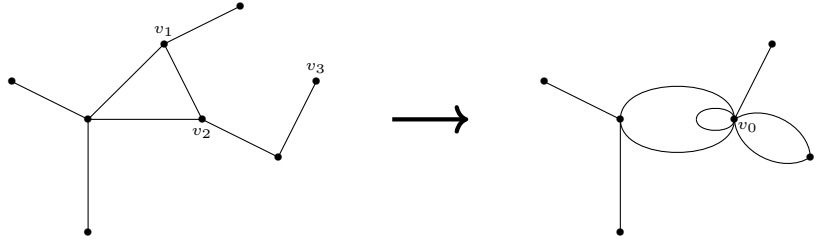
\begin{figure}[h]
    \centering
\begin{tikzpicture}

\node at (-4.5,0) {\tiny$\bullet$};
\node at (-3.5,1) {\tiny$\bullet$};
\node at (-3,0) {\tiny$\bullet$};
\node at (-2,-0.5) {\tiny$\bullet$};
\node at (-1.5,0.5) {\tiny$\bullet$};

\node at (-5.5,0.5) {\tiny$\bullet$};
\node at (-4.5,-1.5) {\tiny$\bullet$};
\node at (-2.5,1.5) {\tiny$\bullet$};

\node at (-3.5,1) [yshift=5pt] {\tiny $v_1$}; 
\node at (-3,0) [yshift=-5pt] {\tiny $v_2$}; 
\node at (-1.5,0.5) [yshift=5pt] {\tiny $v_3$}; 

\draw[] (-4.5,0) -- (-3.5,1)  ;
\draw[] (-4.5,0) -- (-3,0)  ;
\draw[] (-3.5,1) -- (-3,0)  ;
\draw[] (-3,0) -- (-2,-0.5) ;
\draw[] (-2,-0.5) -- (-1.5,0.5) ;

\draw[] (-4.5,0) -- (-5.5,0.5)  ;
\draw[] (-4.5,0) -- (-4.5,-1.5)  ;
\draw[] (-3.5,1) -- (-2.5,1.5)  ;

\draw[->, ultra thick] (-0.5,0) -- (0.5,0)  ;

\node at (2.5,0) {\tiny$\bullet$};
\node at (4,0) {\tiny$\bullet$};
\node at (5,-0.5) {\tiny$\bullet$};

\node at (1.5,0.5) {\tiny$\bullet$};
\node at (2.5,-1.5) {\tiny$\bullet$};
\node at (4.5,1) {\tiny$\bullet$};

\node at (4,0) [xshift=5pt, yshift=-3pt] {\tiny $v_0$}; 

\draw[] (2.5,0) to [bend left=90] (4,0)  ;
\draw[] (2.5,0) to [bend right=90] (4,0)  ;

\draw[] (4,0) to [bend left=60] (5,-0.5)  ;
\draw[] (4,0) to [bend right=60] (5,-0.5)  ;

\draw[] (4,0) to [bend left=90] (3.5,0)  ;
\draw[] (4,0) to [bend right=90] (3.5,0)  ;

\draw[] (2.5,0) -- (1.5,0.5)  ;
\draw[] (2.5,0) -- (2.5,-1.5)  ;
\draw[] (4,0) -- (4.5,1)  ;

\end{tikzpicture}

    \caption{Gluing vertices $v_1,v_2,v_3$ of $\Gamma$ to form a new vertex $v_0$ of $\widetilde\Gamma$.}
    \label{fig:gluingV}
\end{figure}

\begin{theorem}[Gluing vertices]\label{thm:gluingV}
Let Hypothesis \ref{hyp} be satisfied, and let $\widetilde{\Gamma}$ be the graph formed by gluing precisely $k + 1$ vertices of $\Gamma$. Then for the corresponding perturbed Krein Laplacians:
\begin{itemize}
\item[(a)] the positive eigenvalues satisfy the interlacing inequalities
\begin{equation}\label{eq:gluingV+}
\widetilde \lambda_j^+ \leq \lambda_j^+\leq\widetilde\lambda_{j + k}^+ \leq \lambda_{j + k}^+ , \qquad j \in \N;
\end{equation}
\item[(b)] the eigenvalues (counting ground states) satisfy the interlacing inequalities
\begin{equation}\label{eq:gluingV}
 \lambda_j \leq \widetilde\lambda_j \leq\lambda_{j+k} \leq \widetilde\lambda_{j+k} , \qquad j \in \N.
\end{equation}
In particular,
\begin{equation}\label{eq:gluingVV}
0=\lambda_V< \widetilde\lambda_{V-k+1}.
\end{equation}
\end{itemize}
\end{theorem}

\begin{proof}
Denote by $S$ and $\widetilde S$ the symmetric operators in $L^2 (\Gamma)$ and $L^2 (\widetilde \Gamma)$, respectively, defined as in~\eqref{eq:SV}. Then 
\begin{align*}
 \dom S = \widetilde H_0^2 (\Gamma) \subset \widetilde H_0^2 (\widetilde \Gamma) = \dom \widetilde S,
\end{align*}
and the action of the two operators coincides on the smaller domain; we always identify functions on $\Gamma$ with functions on $\widetilde \Gamma$, and conversely, in the obvious way. Thus $S \subset \widetilde S$. 

We show next that the co-dimension of $\dom S$ in $\dom \widetilde S$ is $k$, and we do this for the case $k = 1$ only; for higher $k$ this can be obtained by successively gluing vertices. For $k = 1$, denote by $v_1,v_2$ the vertices of $\Gamma$ that are glued to form the new vertex $v$. Let $f, g \in\dom \widetilde S$ and observe that the linear combination
\begin{equation*}
 h:=(\partial_\nu g (v_1)) f - (\partial_\nu f(v_1)) g
\end{equation*}
satisfies both Dirichlet and Kirchhoff conditions at both $v_1$ and $v_2$, with the latter due to the fact that $f, g$ satisfy Kirchhoff conditions at $v$ (i.e. $\partial_\nu f(v_1) + \partial_\nu f (v_2) = 0$ and likewise for $g$). Then $h \in \dom S$, which proves the claim on the co-dimension. Thus we can apply Theorem~\ref{thm:interlacingAbstract} to obtain inequality \eqref{eq:gluingV+}.

For inequality \eqref{eq:gluingV}, one applies \eqref{eq:Vkernel}, together with the fact that the number of vertices of $\widetilde \Gamma$ is $V - k$, to the chain of inequalities \eqref{eq:gluingV+} to obtain \eqref{eq:gluingV}. Finally, inequality \eqref{eq:gluingVV} is a trivial consequence of \eqref{eq:Vkernel}.
\end{proof}

Gluing vertices therefore increases the eigenvalues of the perturbed Krein Laplacian, with inequality \eqref{eq:gluingV} providing bounds for this increase. Indeed, \eqref{eq:gluingVV} implies that eigenvalues $\lambda_{V-k+1},\dots,\lambda_V$ increase strictly. On the other hand, the increases are counteracted by the fact that the kernel of the operator shrinks after gluing, which explains why the positive eigenvalues actually decrease. By contrast, whilst the eigenvalues of the perturbed standard Laplacian increase by gluing, satisfying in particular the interlacing inequalities
\begin{equation*}
 \lambda_j (- \Delta_{{\rm st},\Gamma,q}) \leq \lambda_j (- \Delta_{{\rm st},\widetilde \Gamma,q})\leq\lambda_{j+k} (- \Delta_{{\rm st},\Gamma,q})\leq \lambda_{j+k} (- \Delta_{{\rm st},\widetilde \Gamma,q}),
\end{equation*}
the kernel is unchanged, and thus the positive eigenvalues increase as well.

\begin{example}\label{ex:glueInterval}
Let $\Gamma=[0,\ell]$ be the interval of length $\ell$. The vertex conditions for the Krein Laplacian $-\Delta_{{\rm K},\Gamma}$ were calculated in Example \ref{ex:interval}. From this, one computes that the eigenvalues $\lambda=\kappa^2$ are given by the solutions of the equation
\begin{equation*}
\left[\cos\frac{\kappa\ell}{2}-\frac{2}{\kappa\ell}\sin\frac{\kappa\ell}{2}\right]\sin\frac{\kappa\ell}{2}=0.
\end{equation*}
The positive solutions to this are
\begin{equation*}
\kappa=\begin{cases}\frac{j\pi}{\ell} & \text{if $j=2,4,6,...$}\\\frac{j\pi}{\ell}-\eta_j & \text{if $j=3,5,7,...$}\end{cases}
\end{equation*}
where the numbers $\eta_j$ are such that $0<\eta_j\ll \frac{\pi}{\ell} $ and $\lim_{j\to\infty}\eta_j=0$.

Now let $\widetilde \Gamma$ be the loop of length $\ell$, formed by gluing together the two vertices of the interval $\Gamma$; see Figure \ref{fig:gluingExample}.
According to Corollary \ref{cor:flower}, the Krein Laplacian $- \Delta_{{\rm K}, \widetilde \Gamma}$ on the loop is identical to the standard Laplacian $- \Delta_{{\rm st}, \widetilde \Gamma}$, and thus they share the same eigenvalues.

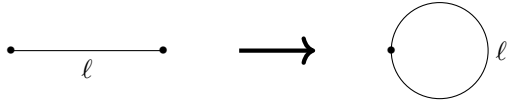
\begin{figure}[h]
    \centering
\begin{tikzpicture}

\node at (-3.5,0) {\tiny$\bullet$};
\node at (-1.5,0) {\tiny$\bullet$};

\node at (-2.5,0) [yshift=-7pt] { $\ell$}; 

\draw[] (-3.5,0) -- (-1.5,0)  ;

\draw[->, ultra thick] (-0.5,0) -- (0.5,0)  ;

\node at (1.5,0) {\tiny$\bullet$};

\node at (2.774,0) [xshift=5pt] { $\ell$}; 

\draw (2.137,0) circle (0.637cm);

\end{tikzpicture}

    \caption{Transforming the interval $\Gamma$ to the loop $\widetilde\Gamma$.}
    \label{fig:gluingExample}
\end{figure}

\noindent The following tables are demonstrative of Theorem \ref{thm:gluingV} for these two graphs: the positive eigenvalues decrease by gluing, but when the ground states are included, they increase.

\begin{table}[!htb]
    \begin{minipage}{.5\linewidth}
      \caption{Positive eigenvalues}
      \centering
\begin{tabular}{c || l | l} 
$j$ & $\lambda_j^+(- \Delta_{{\rm K}, \Gamma})$ & $\lambda_j^+(- \Delta_{{\rm K}, \widetilde \Gamma})$ \\ [0.5ex]
\hline 
&& \\
$1$ & $\left(\frac{2\pi}{\ell}\right)^2$ & $\left(\frac{2\pi}{\ell}\right)^2$ \\ [0.5ex]
$2$ & $\left(\frac{3\pi}{\ell}-\eta_3\right)^2$ & $\left(\frac{2\pi}{\ell}\right)^2$ \\ [0.5ex]
$3$ & $\left(\frac{4\pi}{\ell}\right)^2$ & $\left(\frac{4\pi}{\ell}\right)^2$ \\ [0.5ex]
$4$ & $\left(\frac{5\pi}{\ell}-\eta_5\right)^2$ & $\left(\frac{4\pi}{\ell}\right)^2$ \\ [0.5ex]
$5$ & $\left(\frac{6\pi}{\ell}\right)^2$ & $\left(\frac{6\pi}{\ell}\right)^2$ \\ [0.5ex]
$6$ & $\left(\frac{7\pi}{\ell}-\eta_7\right)^2$ & $\left(\frac{6\pi}{\ell}\right)^2$ 
\end{tabular}
    \end{minipage}%
    \begin{minipage}{.5\linewidth}
      \centering
        \caption{All eigenvalues}
\begin{tabular}{c || l | l} 
$j$ & $\lambda_j(- \Delta_{{\rm K}, \Gamma})$ & $\lambda_j(- \Delta_{{\rm K}, \widetilde \Gamma})$ \\ [0.5ex]
\hline 
&& \\
$1$ & $0$ & $0$ \\ [0.5ex]
$2$ & $0$ & $\left(\frac{2\pi}{\ell}\right)^2$ \\ [0.5ex]
$3$ & $\left(\frac{2\pi}{\ell}\right)^2$ & $\left(\frac{2\pi}{\ell}\right)^2$ \\ [0.5ex]
$4$ & $\left(\frac{3\pi}{\ell}-\eta_3\right)^2$ & $\left(\frac{4\pi}{\ell}\right)^2$ \\ [0.5ex]
$5$ & $\left(\frac{4\pi}{\ell}\right)^2$ & $\left(\frac{4\pi}{\ell}\right)^2$ \\ [0.5ex]
$6$ & $\left(\frac{5\pi}{\ell}-\eta_5\right)^2$ & $\left(\frac{6\pi}{\ell}\right)^2$
\end{tabular}
    \end{minipage} 
\end{table}

\end{example}

\begin{definition}
Assume that Hypothesis \ref{hyp} is satisfied, and let $e_0$ be an edge of $\Gamma$ with (possibly coincident) incident vertices $v_1,v_2$.
Let $\widetilde\Gamma$ be the graph formed from $\Gamma$ by replacing $e_0$ with a path graph from $v_1$ to $v_2$, composed of two edges $e_1,e_2$, joined together by a degree-2 vertex $v_0$, and with total length $\ell(e_1)+\ell(e_2)=\ell(e_0)$. 
Parametrising $e_0$ by $[0,\ell(e_0)]$ and $e_1,e_2$ by $[0,\ell(e_1)],[\ell(e_1),\ell(e_0)]$ respectively, 
where the endpoint $\ell(e_1)$ in both of the latter is identified with $v_0$, the potential $\widetilde q$ associated with $\widetilde\Gamma$ is defined by
\begin{equation*}
\widetilde q_{e_1}:=\left.q_{e_0}\right|_{[0,\ell(e_1)]},\qquad\widetilde q_{e_2}:=\left.q_{e_0}\right|_{[\ell(e_1),\ell(e_2)]}
\end{equation*}
on $e_1,e_2$, and $\widetilde q_e \equiv q_e$ on all other edges $e$. The transformation $- \Delta_{{\rm K}, \Gamma, q}\mapsto- \Delta_{{\rm K}, \widetilde \Gamma, \widetilde q}$ is called 
\emph{inserting a degree-2 vertex along an edge}, and the inverse operation is referred to as \emph{removing a degree-2 vertex}; cf.\ Figure \ref{fig:deg2V}.

\end{definition}

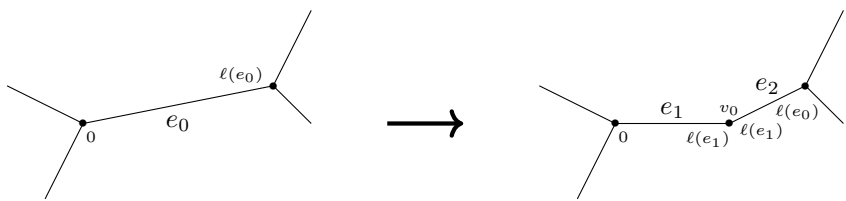
\begin{figure}[h]
    \centering
\begin{tikzpicture}

\node at (7-4.5,0) {\tiny$\bullet$};
\node at (7-3,0) {\tiny$\bullet$};
\node at (7-2,0.5) {\tiny$\bullet$};

\node at (7-3,0) [yshift=5pt] {\tiny $v_0$}; 
\node at (7-3.75,0) [yshift=5pt] { $e_1$}; 
\node at (7-2.5,0.25) [yshift=7pt] { $e_2$}; 

\node at (7-4.5,0) [xshift=3pt, yshift=-5pt] {\tiny $0$}; 
\node at (7-3,0) [xshift=-8pt, yshift=-6pt] {\tiny $\ell (e_1)$}; 
\node at (7-3,0) [xshift=12pt, yshift=-3pt] {\tiny $\ell (e_1)$}; 
\node at (7-2,0.5) [xshift=-3pt, yshift=-10pt] {\tiny $\ell (e_0)$}; 

\draw[] (7-4.5,0) -- (7-3,0)  ;
\draw[] (7-3,0) -- (7-2,0.5) ;

\draw[] (7-4.5,0) -- (7-5.5,0.5)  ;
\draw[] (7-4.5,0) -- (7-5,-1)  ;
\draw[] (7-2,0.5) -- (7-1.5,1.5)  ;
\draw[] (7-2,0.5) -- (7-1.5,0)  ;

\draw[->, ultra thick] (-0.5,0) -- (0.5,0)  ;

\node at (-7+2.5,0) {\tiny$\bullet$};
\node at (-7+5,0.5) {\tiny$\bullet$};

\node at (-7+3.75,0.25) [yshift=-7pt] { $e_0$}; 

\node at (-7+2.5,0) [xshift=3pt, yshift=-5pt] {\tiny $0$}; 
\node at (-7+5,0.5) [xshift=-12pt, yshift=4pt] {\tiny $\ell (e_0)$}; 

\draw[] (-7+2.5,0) -- (-7+5,0.5) ;

\draw[] (-7+2.5,0) -- (-7+1.5,0.5)  ;
\draw[] (-7+2.5,0) -- (-7+2,-1)  ;
\draw[] (-7+5,0.5) -- (-7+5.5,1.5)  ;
\draw[] (-7+5,0.5) -- (-7+5.5,0)  ;

\end{tikzpicture}

    \caption{Inserting a degree-2 vertex $v$ of $\Gamma$ to form $\widetilde\Gamma$.}
    \label{fig:deg2V}
\end{figure}

In the special case that $\Gamma$ is just one loop, it obviously does not make sense to remove the vertex of degree two, as the result would be a graph with one edge but no vertices. However, in this case the above procedure may just be understood as replacing the perturbed Krein Laplacian with the perturbed standard Laplacian. To replace ``Krein vertex conditions'' by standard conditions on arbitrary vertices, we refer to Theorem \ref{thm:addstandardB}.

\begin{theorem}[Inserting degree-2 vertices]\label{thm:deg2V}
Let Hypothesis \ref{hyp} be satisfied, and let $\widetilde{\Gamma}$ be the graph formed by inserting $k_0$ vertices of degree 2 along edges of $\Gamma$. Then for the corresponding perturbed Krein Laplacians:
\begin{itemize}
\item[(a)] the positive eigenvalues satisfy 
\begin{equation}\label{eq:deg2V+}
 \lambda_j^+ \leq\widetilde\lambda_j^+  \leq \lambda_{j + k_0}^+ \leq \widetilde\lambda_{j + k_0}^+, \qquad j \in \N;
\end{equation}
\item[(b)] the eigenvalues (counting ground states) satisfy
\begin{equation}\label{eq:deg2V}
\widetilde\lambda_{j}\leq \lambda_{j}\leq\widetilde \lambda_{j + k_0} \leq \lambda_{j + k_0} , \qquad j \in \N.
\end{equation}
\end{itemize}
\end{theorem}

\begin{proof}
If we define $S$ and $\widetilde S$ corresponding to $\Gamma$ and $\widetilde \Gamma$ respectively, as in \eqref{eq:SV}, then $\widetilde S \subset  S$. Moreover, $\dom \widetilde S$ has co-dimension $k_0$ in $\dom  S$; indeed, if $k_0 = 1$, then for any two linearly independent functions $f, g \in \dom  S$, the function $f (v_1) g - g (v_1) f$ vanishes at $v_1$ and thus belongs to $\dom\widetilde S$. The case of arbitrary $k_0$ follows inductively. Then all estimates in \eqref{eq:deg2V+} follow directly from Theorem \ref{thm:interlacingAbstract}, noting that the roles of $S$ and $\widetilde S$ are reversed. After this, \eqref{eq:deg2V} follows with the help of \eqref{eq:Vkernel}. 
\end{proof}

The following example shows that the positive eigenvalues of the Krein Laplacian may indeed increase strictly from adding a degree-2 vertex, in contrast with the standard Laplacian which does not feel degree-2 vertices at all.

\begin{example}
Let $\Gamma$ be the interval of length two, and let $\widetilde\Gamma$ be the path graph formed by inserting a vertex of degree 2 at its midpoint, creating two intervals each of length one connected by a single vertex.
%Let $\Gamma$ be the path graph consisting of one vertex of degree two and two edges of length one, each of which connects this vertex to a vertex of degree one. Removing the middle vertex of degree two leads to an interval $\widetilde \Gamma$ of length two. 
A direct computation shows that the positive eigenvalues of the Krein Laplacian on $\widetilde\Gamma$ are the numbers $\kappa^2$ for which $\kappa$ is a root of
\begin{align*}
 \kappa \big( (\kappa^2 - 2) \sin (2 \kappa) + \kappa + 4 \sin \kappa - 4 \kappa \cos \kappa + 3 \kappa \cos(2 \kappa) \big) = 0.
\end{align*}
The lowest two positive eigenvalues are then $\widetilde\lambda_1^+ \approx 4.5^2$ and $\widetilde\lambda_2^+ = (2 \pi)^2$. In contrast to this, the first two positive eigenvalues of the Krein Laplacian on $\Gamma$ are $\lambda_1^+=\pi^2$ and $\lambda_2^+<(3 \pi/2)^2$; cf.\ Example \ref{ex:glueInterval}.
\end{example}

We have seen in Theorem \ref{thm:gluingV} how the eigenvalues change upon gluing vertices of~$\Gamma$. It is also possible to glue arbitrary points of $\Gamma$ together. Again, as the (perturbed) Krein Laplacian distinguishes between vertices of degree two and non-vertex points on the graph, the following is more general than Theorem \ref{thm:gluingV}.

\begin{definition}
Assume that Hypothesis \ref{hyp} is satisfied, and let $\mathcal{N}$ be a finite subset of points in $\Gamma$ (which may include both vertices and points along edges). 
Let $\widetilde\Gamma$ be the graph formed by first inserting a vertex at each of the points in $\mathcal{N}$ which are not already vertices, 
and then gluing all of these new vertices together with the remaining vertices in $\mathcal{N}$ to form a single point.
The transformation $- \Delta_{{\rm K}, \Gamma, q}\mapsto- \Delta_{{\rm K}, \widetilde \Gamma, q}$ is called \emph{gluing the points in $\cN$}.
\end{definition}

This is evidently a two-step process, consisting of insterting degree-2 vertices along edges, and then gluing vertices.
In general, one cannot determine the effect on individual eigenvalues since they increase during the first step but decrease during the second.
Nevertheless, a direct application of Theorems \ref{thm:gluingV} and \ref{thm:deg2V} gives some insight into their behaviour. 

\begin{corollary}[Gluing arbitrary points]\label{cor:arbgluingV}
Assume that Hypothesis \ref{hyp} is satisfied. Let $\mathcal{N}$ be a finite subset of $k+1$ points in $\Gamma$ of which $k_0\leq k+1$ are not vertices, and let $\widetilde{\Gamma}$ be the graph formed by gluing these points together. Then for the corresponding perturbed Krein Laplacians:
\begin{itemize}
\item[(a)] the positive eigenvalues satisfy
\begin{align*}
 \widetilde \lambda_{j}^+ \leq \lambda_{j + k_0}^+ \leq \widetilde\lambda_{j + k + k_0}^+  \leq \lambda_{j + k + 2 k_0}^+, \qquad j \in \N;
\end{align*}
\item[(b)] the eigenvalues (counting ground states) satisfy
\begin{align*}
 \widetilde \lambda_{j}  \leq \lambda_{j + k}  \leq \widetilde\lambda_{j + k + k_0} \leq \lambda_{j + 2 k + k_0} , \qquad j \in \N.
\end{align*}
\end{itemize}
\end{corollary}

Next, we move on to transformations which change the volume of $\Gamma$. Here, the potential $q$ will not be well-defined on the new graph, for which the associated potential $\widetilde q$ is defined accordingly.

\begin{definition}
Assume that Hypothesis \ref{hyp} is satisfied. Let $\widetilde\Gamma$ be the graph formed from $\Gamma$ by lengthening one of its edges, $e_0$, by a factor of $\alpha>1$, so that it has length $\widetilde \ell (e_0) = \alpha \ell (e_0)$ in $\widetilde \Gamma$. If there is a potential $q$ associated with~$\Gamma$, then we define the potential $\widetilde q$ associated with $\widetilde\Gamma$ via
\begin{equation}\label{lengthenedpotential}
 \widetilde q_{e_0}(x):=\alpha^{-2}q_{e_0} (x/\alpha),
\end{equation}
and $\widetilde q_e \equiv q_e$ on all other edges. The transformation $- \Delta_{{\rm K}, \Gamma, q}\mapsto- \Delta_{{\rm K}, \widetilde \Gamma, \widetilde q}$ is called 
\emph{lengthening the edge $e_0$}, and the inverse operation is referred to as \emph{shrinking the edge $e_0$}.
\end{definition}

\begin{theorem}[Lengthening an edge]\label{thm:lengthendgeV}
Let Hypothesis \ref{hyp} be satisfied, and let $\widetilde{\Gamma}$ be the graph formed by lengthening one of the edges of $\Gamma$. Then for the corresponding perturbed Krein Laplacians:
\begin{itemize}
\item[(a)] the positive eigenvalues satisfy
\begin{equation}\label{eq:lengthenedgeV+}
\widetilde \lambda_j^+\leq\lambda_j^+, \qquad j \in \N;
\end{equation}
\item[(b)] the eigenvalues (counting ground states) satisfy
\begin{equation}\label{eq:lengthenedgeV}
\widetilde \lambda_j \leq  \lambda_j , \qquad j \in \N.
\end{equation}
\end{itemize}
\end{theorem}

\begin{proof}
Suppose that an edge $e_0$ of $\Gamma$ is lengthened by a factor of $\alpha>1$. Given $f\in\widetilde{H}_0^2(\Gamma)$, 
let  $\widetilde{f}$ be the function such that $\widetilde{f}_{e_0} (x) = \alpha f_{e_0} (x/\alpha)$ and  $\widetilde{f}_e (x) = f_e (x)$ for all other edges $e$. 
Now, $\widetilde{f}_{e_0}(0)=\widetilde{f}_{e_0}( \ell (e_0))=0$, preserving the Dirichlet conditions, 
and $\widetilde{f}_{e_0}'(0)=f_{e_0}'(0)$, $\widetilde{f}_{e_0}'( \ell (e_0)) = f_{e_0}' (\ell (e_0))$, preserving the Kirchhoff conditions, 
whence $\widetilde{f}\in\widetilde{H}_0^2(\widetilde{\Gamma})$. 
 Notice that $\widetilde{H}_0^2({\Gamma})\to\widetilde{H}_0^2(\widetilde{\Gamma}):f\mapsto\widetilde{f}$ is a bijection. Then
\begin{align*}
\int_{0}^{\widetilde \ell (e_0)} |\widetilde{f}_{e_0}'|^2 \dd x + \int_{0}^{\widetilde \ell (e_0)}\widetilde{q}_{e_0}|\widetilde{f}_{e_0}|^2\dd x &= \alpha \int_0^{\ell (e_0)} |f_{e_0}'|^2\dd x + \alpha \int_0^{\ell (e_0)} q_{e_0} |f_{e_0}|^2 \dd x,
\end{align*}
and
\begin{align*}
 \int_{0}^{\widetilde \ell (e_0)} \left|-\widetilde{f}_{e_0}'' + \widetilde{q}_{e_0}\widetilde{f}_{e_0} \right|^2 \dd x &= \frac{1}{\alpha} \int_0^{\ell (e_0)} \left|-f_{e_0}'' + q_{e_0} f_{e_0} \right|^2 \dd x,
\end{align*}
recalling that the potential is redefined by \eqref{lengthenedpotential} on the lengthened edge. Thus
\begin{equation*}
\frac{\int_{\widetilde{\Gamma}}|-\widetilde{f}''+\widetilde{q}\widetilde{f}|^2 \dd x}{\int_{\widetilde{\Gamma}}|\widetilde{f}'|^2\dd x + \int_{\widetilde{\Gamma}}\widetilde{q}|\widetilde{f}|^2\dd x}\leq\frac{ \int_\Gamma |-f''+qf|^2 \dd x}{\int_\Gamma |f'|^2 \dd x+\int_\Gamma q|f|^2 \dd x}.
\end{equation*}
Inequality \eqref{eq:lengthenedgeV+} follows from Theorem \ref{thm:minMaxKreinV}, and then \eqref{eq:lengthenedgeV} from \eqref{eq:Vkernel} since the kernel of the operator is unchanged by the transformation.
\end{proof}

The remaining surgery operation deals with expanding the graph by inserting a new finite, connected metric graph $\Gamma_0$ in some way to the original graph. 
If there is a potential $q_0$ associated with $\Gamma_0$, then we assume that it satisfies the following hypothesis in agreement with what is assumed for $q$ on $\Gamma$.

\begin{hypothesis}\label{hyp0}
On the finite, connected metric graph $\Gamma_0$, the potential $q_0 : \Gamma_0 \to \R$ is measurable and bounded, and $q_0 (x) \geq 0$ holds for almost all $x \in \Gamma_0$.
\end{hypothesis}

As a rule, if no new potential is specified on the new edges, then it is reasonable to take the potential to be zero there. Nevertheless, the inequalities in Theorem~\ref{thm:attachV} hold for the potential chosen arbitrarily there under Hypothesis~\ref{hyp0}.

\begin{definition}
Let Hypothesis \ref{hyp} be satisfied, and let $\widetilde\Gamma$ be the graph formed from $\Gamma$ by gluing $m$ of the vertices of a finite, connected metric graph $\Gamma_0$ to distinct vertices of $\Gamma$.
The new potential $\widetilde q$ associated with $\widetilde\Gamma$ is identical to $q$ on the edges inherited from $\Gamma$ and satisfies Hypothesis \ref{hyp0} on the edges from $\Gamma_0$. The transformation $- \Delta_{{\rm K}, \Gamma, q}\mapsto- \Delta_{{\rm K}, \widetilde \Gamma, \widetilde q}$ is called \emph{attaching a (connected) graph to $\Gamma$ (by $m$ vertices)}. The inverse operation may be referred to as \emph{deleting a (connected) subgraph}; cf.\ Figure \ref{fig:attachV}.
\end{definition}

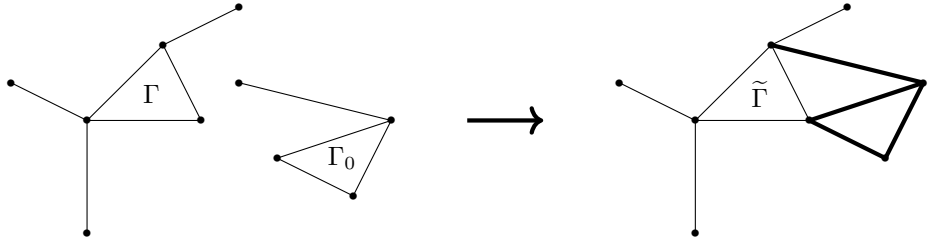
\begin{figure}[h]
    \centering
\begin{tikzpicture}

\node at (-5.5,0) {\tiny$\bullet$};
\node at (-4.5,1) {\tiny$\bullet$};
\node at (-4,0) {\tiny$\bullet$};
\node at (-6.5,0.5) {\tiny$\bullet$};
\node at (-5.5,-1.5) {\tiny$\bullet$};
\node at (-3.5,1.5) {\tiny$\bullet$};

\node at (-3,-0.5) {\tiny$\bullet$};
\node at (-2,-1) {\tiny$\bullet$};
\node at (-1.5,0) {\tiny$\bullet$};
\node at (-3.5,0.5) {\tiny$\bullet$};

\node at (-4.65,0) [yshift=10pt] {$\Gamma$}; 
\node at (-2.25,-0.5) [xshift=3pt] {$\Gamma_0$}; 

\draw[] (-5.5,0) -- (-4.5,1)  ;
\draw[] (-5.5,0) -- (-4,0)  ;
\draw[] (-4.5,1) -- (-4,0)  ;
\draw[] (-5.5,0) -- (-6.5,0.5)  ;
\draw[] (-5.5,0) -- (-5.5,-1.5)  ;
\draw[] (-4.5,1) -- (-3.5,1.5)  ;

\draw[] (-3,-0.5) -- (-2,-1) ;
\draw[] (-2,-1) -- (-1.5,0) ;
\draw[] (-3,-0.5) -- (-1.5,0) ;
\draw[] (-1.5,0) -- (-3.5,0.5) ;

\draw[->, ultra thick] (-0.5,0) -- (0.5,0)  ;

\node at (2.5,0) {\tiny$\bullet$};
\node at (3.5,1) {\tiny$\bullet$};
\node at (4,0) {\tiny$\bullet$};
\node at (1.5,0.5) {\tiny$\bullet$};
\node at (2.5,-1.5) {\tiny$\bullet$};
\node at (4.5,1.5) {\tiny$\bullet$};

\node at (5,-0.5) {\tiny$\bullet$};
\node at (5.5,0.5) {\tiny$\bullet$};

\node at (3.35,0) [yshift=10pt] {$\widetilde\Gamma$}; 

\draw[] (2.5,0) -- (3.5,1)  ;
\draw[] (2.5,0) -- (4,0)  ;
\draw[] (3.5,1) -- (4,0)  ;
\draw[] (2.5,0) -- (1.5,0.5)  ;
\draw[] (2.5,0) -- (2.5,-1.5)  ;
\draw[] (3.5,1) -- (4.5,1.5)  ;

\draw[ultra thick] (4,0) -- (5,-0.5)  ;
\draw[ultra thick] (5,-0.5) -- (5.5,0.5)  ;
\draw[ultra thick] (4,0) -- (5.5,0.5)  ;
\draw[ultra thick] (5.5,0.5) -- (3.5,1) ;

\end{tikzpicture}

    \caption{Attaching $\Gamma_0$ to $\Gamma$ by two vertices to form $\widetilde\Gamma$. The new edges in $\widetilde\Gamma$ are shown in bold.}
    \label{fig:attachV}
\end{figure}

\begin{theorem}[Attaching a graph]\label{thm:attachV}
Assume that Hypotheses \ref{hyp} and \ref{hyp0} hold. Let $\widetilde{\Gamma}$ be the graph formed by attaching $\Gamma_0$ to $\Gamma$ by $m$ vertices. Then for the corresponding perturbed Krein Laplacians: 
\begin{itemize}
\item[(a)] the positive eigenvalues satisfy
\begin{equation}\label{eq:attachV+}
 \widetilde\lambda_j^+\leq\lambda_j^+, \qquad j \in \N;
\end{equation}
\item[(b)] the eigenvalues (counting ground states) satisfy
\begin{equation}\label{eq:attachV}
\widetilde\lambda_{j + V_0 - m} \leq \lambda_j, \qquad j \in \N;
\end{equation}
here $V_0$ is the number of vertices of $\Gamma_0$.
\end{itemize}
\end{theorem}

\begin{proof}
Every function in $\widetilde H_0^2 (\Gamma)$ can be extended by zero to a function in $\widetilde H_0^2 (\widetilde{\Gamma})$, and this does not change the Rayleigh quotient. Thus inequality \eqref{eq:attachV+} follows from Theorem~\ref{thm:minMaxKreinV}. Finally, \eqref{eq:attachV} is obtained from  \eqref{eq:Vkernel}, since the dimension of the kernel of the operator increases by $V_0 - m$.
\end{proof}

A special case of the previous theorem consists of inserting a single edge between two vertices of $\Gamma$, a process which does not change the dimension of the kernel of the perturbed Krein Laplacian.

\begin{corollary}[Inserting an edge between existing vertices]\label{cor:addedgeV}
Let Hypothesis \ref{hyp} hold, and let $\widetilde{\Gamma}$ be the graph formed by inserting an edge $e_0$ between two (not necessarily distinct) vertices of $\Gamma$. Assume that the potential $q_0$ on $\Gamma_0=e_0$ satisfies Hypothesis \ref{hyp0}. Then the eigenvalues (counting ground states) of the corresponding perturbed Krein Laplacians satisfy
\begin{align*}
\widetilde \lambda_j \leq  \lambda_j , \qquad j \in \N.
\end{align*}
\end{corollary}

We emphasise that this behaviour differs substantially from the one for standard vertex conditions, where inserting an edge may either increase or decrease eigenvalues; cf.\ \cite{KMN13}.

\section{Isoperimetric inequalities}\label{sec:isoperimetric}

We now turn to estimates for the positive eigenvalues of the perturbed Krein Laplacian. We start with a lower estimate for the first positive eigenvalue, which we may call the {\em spectral gap}; cf.\ Remark \ref{rem:gap} below.

\begin{theorem}\label{thm:upperBound}
Assume Hypothesis \ref{hyp}, and denote by $\ell (\Gamma)$ the total length of $\Gamma$. Furthermore, let $\Lambda$ be the loop of length $\ell (\Gamma)$. If $q$ is not identically zero on $\Gamma$ then
\begin{align}\label{eq:lowerBoundV}
 \lambda_1^+ (- \Delta_{{\rm K}, \Gamma, q}) > \lambda_1^+ (- \Delta_{\delta, \Lambda, I})
\end{align}
holds, where $- \Delta_{\delta,\Lambda, I}$ is the Laplacian on $\Lambda$ with a $\delta$-interaction of strength $I := \int_\Gamma q \,\dd x$ at one (arbitrary) point. In the potential-free case,
\begin{equation}\label{eq:1stlowerBound}
 \lambda_1^+ (- \Delta_{{\rm K}, \Gamma}) \geq \lambda_1^+ (- \Delta_{\rm st, \Lambda}) = 4 \left(\frac{\pi}{\ell (\Gamma)}\right)^2.
\end{equation}
Equality in \eqref{eq:1stlowerBound} holds if and only if $\Gamma$ is an interval, a loop, an equilateral 2-cycle, or an equilateral figure-8.
\end{theorem}

\begin{proof}
For any potential satisfying Hypothesis \ref{hyp}, let $\widetilde \Gamma$ be the flower graph formed from $\Gamma$ by gluing all vertices. Then by Theorem~\ref{thm:gluingV} and Theorem~\ref{thm:inequalitiesV}, we have
\begin{align}\label{eq:a}
 \lambda_1^+ (- \Delta_{{\rm K}, \Gamma, q}) \geq \lambda_1^+ (- \Delta_{{\rm K}, \widetilde \Gamma, q}) \geq \lambda_1^+ (- \Delta_{{\rm st}, \widetilde \Gamma, q}).
\end{align}
Moreover, as the only vertex of $\widetilde \Gamma$ has even degree (equal to twice the number of edges), we may cut through the vertex in such a way that we obtain an (Eulerian) cycle~$\Lambda$ of length $\ell (\Gamma)$, and by surgery principles for the perturbed standard Laplacian $- \Delta_{{\rm st}, \widetilde \Gamma, q}$, see e.g. \cite[Theorem~4.1]{RS20},  we get
\begin{align}\label{eq:b}
 \lambda_1^+ (- \Delta_{{\rm st}, \widetilde \Gamma, q}) \geq \lambda_1^+ (- \Delta_{{\rm st}, \Lambda, q}).
\end{align}
In the case $q = 0$ identically, we now obtain \eqref{eq:1stlowerBound} from a direct calculation. If $q$ is nontrivial, then $I > 0$, and for both operators $- \Delta_{{\rm st}, \Lambda, q}$ and $- \Delta_{\delta, \Lambda, I}$, the smallest eigenvalue is positive. Hence we may argue further as in the proof of~\cite[Theorem~1]{KKT16}: let $\psi$ be an eigenfunction of $- \Delta_{{\rm st}, {\Lambda}, q}$ corresponding to its lowest eigenvalue. Then
\begin{align}\label{eq:extra}
 \lambda_1^+ (- \Delta_{{\rm st}, \Lambda, q}) = \frac{\int_{\Lambda} |\psi'|^2  \dd x + \int_{\Lambda} q |\psi|^2 \dd x}{\int_{\Lambda} |\psi|^2 \dd x} \geq \frac{\int_{\Lambda} |\psi'|^2  \dd x + \int_{\Lambda} q |\psi (x_{\min})|^2 \dd x}{\int_{\Lambda} |\psi|^2 \dd x},
\end{align}
where $x_{\min}$ is any point on ${\Lambda}$ where $|\psi|$ takes its minimum. Since the last quotient is the Rayleigh quotient of the Laplacian with a $\delta$-vertex condition of strength $I = \int_\Gamma q \dd x$ at $x_{\min}$, it follows that $\lambda_1^+ (- \Delta_{{\rm K}, \Gamma, q}) \geq \lambda_1^+ (- \Delta_{\delta, \Lambda, I})$. However, equality would imply equality in \eqref{eq:extra}, which in turn would yield $|\psi| = |\psi (x_{\rm min})|$ almost everywhere on the support of $q$. At the same time, equality also implies that $\psi$ is an eigenfunction of $- \Delta_{\delta, \Lambda, I}$, and then $\psi = 0$ on the support of $q$; since $\Lambda$ is a loop, this would imply $\psi = 0$ identically, a contradiction. This proves~\eqref{eq:lowerBoundV}.

In the case of equality in~\eqref{eq:1stlowerBound}, all of the inequalities in \eqref{eq:a} and \eqref{eq:b} must in fact be equalities. In particular, the standard Laplacian on the flower graph $\widetilde \Gamma$ in the above argument already has to have $4 \pi^2 / \ell (\widetilde \Gamma)^2$ as its first positive eigenvalue, which is only possible if on the loop $\Lambda$ resulting from splitting the central vertex of $\widetilde \Gamma$, there exists an eigenfunction for the first positive eigenvalue which has the same value at each point that was glued together previously (cf.\ \cite[Theorem 1]{KMN13}). Since each eigenfunction of $- \Delta_{{\rm st}, \Lambda}$ corresponding to the first nonzero eigenvalue takes each of its values exactly twice on the loop --- at two points with distance $\ell (\Gamma) / 2$ from each other --- it follows that $\widetilde \Gamma$ can be recovered from $\Lambda$ by gluing at most two points. Hence $\widetilde \Gamma$ is either a loop itself or an equilateral figure-8. In other words, joining all vertices in the original graph $\Gamma$ leads to a loop or a figure-8, and this is only possible if $\Gamma$ is of one of the following six types: an interval, a path graph with two equal edges, a loop, an equilateral 2-cycle or an equilateral figure-8. Considering these graphs only, one finds by calculation that there exist eigenfunctions with corresponding eigenvalue $4 \pi^2 / \ell (\Gamma)^2$ if and only if $\Gamma$ is equilateral and has one of the four forms listed in the statement of the theorem. 
\end{proof}

\begin{remark}\label{rem:gap}
The interval $(0, 4 \pi^2 / \ell (\Gamma)^2)$ has empty intersection not only with the spectrum of the Krein Laplacian on one individual graph $\Gamma$. In fact, Theorem~\ref{thm:upperBound} asserts that, for fixed $\ell > 0$, the interval $(0, 4 \pi^2 / \ell^2)$ is free of spectrum for the Krein Laplacians on the whole class of metric graphs with total length $\ell$.
\end{remark}

\begin{remark}
Alternatively, one may use \eqref{eq:Friedlander} in combination with known lower bounds on the eigenvalues of the standard Laplacian to obtain lower bounds for the positive Krein Laplacian eigenvalues. However, using the optimal lower bound from \cite{F05}, one gets
\begin{align*}
 \lambda_1^+ (- \Delta_{{\rm K}, \Gamma}) \geq \lambda_2 (- \Delta_{{\rm st}, \Gamma}) \geq \frac{\pi^2}{\ell (\Gamma)^2},
\end{align*}
which is weaker than the sharp bound \eqref{eq:1stlowerBound}.
\end{remark}

\begin{remark}
The two crucial surgery operations used in the above proof are standard: gluing all vertices of a graph into one was used in \cite{KKMM16}, and cutting through vertices to obtain an Eulerian cycle goes back at least to \cite{N82}; see also \cite{KN14}. Nevertheless, the above proof is slightly unusual: for the standard Laplacian, gluing vertices increases eigenvalues (the positive ones, as well as counting the ground state) whilst cutting vertices decreases them, so that both surgery operations used above --- gluing all vertices into one and cutting vertices to obtain an Eulerian cycle --- cannot be used within the same argument. However, in the present situation this works smoothly since gluing is performed on the positive eigenvalues of the perturbed Krein Laplacian and cutting is done only after transition to standard vertex conditions.
\end{remark}

We point out that the exact same proof also yields an estimate for higher eigenvalues in the potential-free case:

\begin{theorem}
Assume that Hypothesis \ref{hyp} is satisfied with $q = 0$ identically, and that $\Lambda$ is a loop with the same length as for $\Gamma$. Then
\begin{align*}
 \lambda_j^+ (- \Delta_{{\rm K}, \Gamma}) \geq \lambda_j^+ (- \Delta_{\rm st, \Lambda}) = \lambda_{j + 1} (- \Delta_{\rm st, \Lambda})
\end{align*}
holds for all $j \in \N$.
\end{theorem}

We conclude this section with a remark on how to apply the min-max principle to get upper spectral bounds. We do not go far into this and discuss only, very briefly, the special case of graphs which contain Eulerian cycles. We restrict ourselves here to the potential-free case, although natural generalisations for potentials exist (but their formulation may be less pleasant).

\begin{remark}
Suppose that $\Gamma$  contains an Eulerian cycle $\Sigma$ (obtained by cutting through vertices and removing edges not on the cycle), and let $\cE_\Sigma\subseteq\cE$ denote the set of edges belonging to $\Sigma$. Then the function $f$ which on each $e\in\cE_\Sigma$ takes the form
\begin{equation*}
 f_{e}(x)=\pm\frac{\ell(e)}{n_e}\sin\left(\frac{n_e\pi x}{\ell(e)}\right),\qquad x\in [0,\ell(e)],
\end{equation*}
for some $n_e\in\mathbb{N}$, clearly satisfies Dirichlet conditions at all vertices of $\Sigma$, and, moreover, its derivatives have equal magnitude at all endpoints. Each $f_e$ contains $n_e/2$ periods of sine, and thus, by moving around the cycle, one can ensure that Kirchhoff conditions are satisfied at all vertices of $\Sigma$ by choosing appropriate signs for $f_e$ on adjacent pairs of edges; the only place where there could be a discrepancy is when one returns to the start of the cycle, as the function may end on a half-number of periods, but this problem is averted by imposing the further restriction that $\sum_{e\in\cE_\Sigma}n_e\in 2\mathbb{N}$. Now, $f$ satisfies Dirichlet-Kirchhoff conditions not only on $\Sigma$, but also on $\Gamma$, after extending it by zero on $\cE\backslash\cE_\Sigma$, so such functions provide upper estimates for the positive eigenvalues of $-\Delta_{{\rm K},\Gamma}$ via the min-max principle, Theorem \ref{thm:minMaxKreinV}. The Rayleigh quotient for this $f$ is
\begin{equation*}
R_{\rm K}[f]=\frac{\pi^2}{\ell (\Sigma)}\sum_{e\in\cE_\Sigma}\frac{n_e^2}{\ell(e)},
\end{equation*}
which is an explicit upper bound for the first positive eigenvalue; the maximum value of $R_{\rm K}[f]$ among $j$ linearly independent functions of this type gives an upper estimate for $\lambda_j^+ (-\Delta_{{\rm K},\Gamma})$.

Of course, it is true in general, even with potentials, that for $\Gamma$ containing an Eulerian cycle $\Sigma$, one has $\lambda_j^+ (-\Delta_{{\rm K},\Gamma,q}) \leq \lambda_j^+ (-\Delta_{{\rm K},\Sigma,\widetilde q})$, where $\widetilde q:=\left.q\right|_\Sigma$, due to Theorems \ref{thm:gluingV} and \ref{thm:attachV}.
\end{remark}

\section{More general perturbed Krein Laplacians}\label{sec:moreGeneral}

Thus far, we have studied the Krein extension of the symmetric perturbed Laplacian with Dirichlet and Kirchhoff conditions at all vertices, but the abstract theory of Krein extensions of symmetric operators allows one to extend this work 
to cover symmetric perturbed Laplacians with more general vertex conditions. In this section we illustrate this by considering perturbed Laplacians with ``Krein vertex conditions'' on a selected subset of the vertex set, and standard (continuity-Kirchhoff) vertex conditions at all further vertices. We indicate in which form the results of the previous sections carry over to this setting. The proofs are analogous in the present case and are mostly left to the reader.

Let Hypothesis \ref{hyp} be satisfied. For $\cB \subset \cV$, define the operator $S_\cB$ in $L^2 (\Gamma)$ by
\begin{align}\label{eq:SB}
\begin{split}
 (S_\cB f)_e & = - f_e'' + q_e f_e \qquad \text{on each edge}~e \in \cE, \\
 \dom S_\cB & = \Big\{ f \in \widetilde H^2 (\Gamma) \cap H^1 (\Gamma) : \partial_\nu f (v) = 0~\text{for each}~v \in \cV, \\
 & \qquad \qquad \qquad \qquad \qquad \quad f (v) = 0~\text{for each}~v \in \cB \Big\}.
\end{split}
\end{align}

\begin{remark}
A more general setting may be treated with the same methods, but we do not go into these details here: it is possible to replace the standard vertex conditions at the vertices in $\cV \setminus \cB$ by any self-adjoint, local vertex conditions. For the description of such conditions, we refer the reader to \cite{BK12}.
\end{remark}

\begin{remark}
The reader may think of the selected vertex set $\cB$ as a kind of boundary for $\Gamma$. One choice, which may be natural in some cases, is to let $\cB$ consist of all vertices of degree one. We are not restricted to this situation, but we may keep it in mind as a typical example.
\end{remark}

The operator $S_\cB$ in \eqref{eq:SB} is symmetric, closed, and densely defined. It has defect numbers $n_- = n_+ = |\cB|$, and is thus only self-adjoint if $\cB = \emptyset$. Furthermore, $S_\cB$ is clearly nonnegative, and its Friedrichs extension $S_{\cB, \rm F}$ is the perturbed Laplacian subject to Dirichlet boundary conditions on $\cB$ and standard vertex conditions on $\cV \setminus \cB$. In particular,
\begin{align}\label{eq:muB}
 (S_\cB f, f) \geq \mu \|f\|^2, \quad f \in \dom S_\cB
\end{align}
holds, where $\mu > 0$ may be chosen as the lowest eigenvalue of $S_{\cB, \rm F}$. The adjoint of $S_\cB$ equals
\begin{align*}
 (S_\cB^* f)_e & = - f_e'' + q_e f_e \qquad \text{on each edge}~e \in \cE, \\
 \dom S^* & = \left\{ f \in \widetilde H^2 (\Gamma) \cap H^1 (\Gamma) : \partial_\nu f (v) = 0~\text{for all}~v \in \cV \setminus \cB \right\}.
\end{align*}

Due to \eqref{eq:muB}, for nonempty $\cB \subset \cV$, we may consider the operator
\begin{align*}
 - \Delta_{{\rm K}, \Gamma, q, \cB} := S_{\cB, \rm K},
\end{align*}
the Krein--von Neumann extension of $S_\cB$. If $q = 0$ identically, we write $- \Delta_{{\rm K}, \Gamma, \cB}$.

To derive some properties of the operator $- \Delta_{{\rm K}, \Gamma, q, \cB}$, constructing an appropriate boundary triple is useful.

\begin{proposition}\label{prop:graphBTB}
Assume that Hypothesis \ref{hyp} is satisfied, and let $\cB \subset \cV$ be nonempty. Let $S_\cB$ be defined in \eqref{eq:SB}. For $f \in \dom S_\cB^*$, define
\begin{align*}
 \Gamma_0 f = \begin{pmatrix} f (v_1) \\ \vdots \\ f (v_b) \end{pmatrix} \quad \text{and} \quad \Gamma_1 f = \begin{pmatrix} - \partial_\nu f (v_1) \\ \vdots \\ - \partial_\nu f (v_b) \end{pmatrix},
\end{align*}
where $\cB = \{v_1, \dots, v_b\}$ (and $b = |\cB|$). Then $\{ \C^b, \Gamma_0, \Gamma_1\}$ is a boundary triple for $S_\cB^*$; in particular, $S_\cB$ has defect numbers
\begin{align*}
 n_- = n_+ = b.
\end{align*}
The corresponding extensions $A$ and $B$ of $S$ defined in~\eqref{eq:AB} are given by
\begin{align*}
 A = S_{\cB, \rm F} \quad \text{and} \quad B = - \Delta_{{\rm st}, \Gamma, q};
\end{align*}
in particular, $0 \in \rho (A)$. The value of the corresponding Weyl function at $\lambda = 0$ is $M_\cB (0) = - \Lambda_{q, \cB}$, where $\Lambda_{q, \cB}$ is the {\em Dirichlet-to-Neumann matrix for $\cB$} defined via the relation
\begin{align*}
 \begin{pmatrix} \partial_\nu f_* (v_1) \\ \vdots \\ \partial_\nu f_* (v_b) \end{pmatrix} = \Lambda_{q, \cB} \begin{pmatrix} f_* (v_1) \\ \vdots \\ f_* (v_b) \end{pmatrix},
\end{align*}
where $f_* \in \ker S_\cB^*$ is arbitrary. 
\end{proposition}

\begin{remark}
The Weyl function $\lambda \mapsto M_\cB (\lambda)$ may be computed from the Weyl function $\lambda \mapsto M (\lambda)$ of the boundary triple in Proposition \ref{prop:graphBT}. Indeed, if we write $\cV = \{v_1, \dots, v_b, v_{b + 1}, \dots, v_V\}$, where the vertices are ordered such that the first $b$ of them form $\cB$, and write
\begin{align*}
 M (\lambda) = \begin{pmatrix} \widehat D (\lambda) & - B (\lambda)^\top \\ - B (\lambda) & \widehat L (\lambda) \end{pmatrix},
\end{align*}
where the block decomposition is taken according to the decomposition of the vertices into $\cB$ and $\cV \setminus \cB$, then we have
\begin{align*}
 M_\cB (\lambda) = \widehat D (\lambda) - B (\lambda)^\top \widehat L (\lambda)^{-1} B (\lambda).
\end{align*}
The proof is straightforward; for a special case it may be found in \cite[Proposition~3.1]{GR20}; see also \cite[Lemma 3.1]{KL20}. In particular, in the potential-free case, $- M_\cB (0)$ is the Schur complement of the weighted discrete Laplacian $L$ in \eqref{eq:L} with respect to decomposition of the vertices into $\cB$ and $\cV \setminus \cB$.
\end{remark}

From Proposition \ref{prop:graphBTB}, the following properties of $- \Delta_{{\rm K}, \Gamma, q, \cB}$ are immediate.

\begin{proposition}\label{prop:KreinB}
Assume that Hypothesis \ref{hyp} holds and that $\cB \subset \cV$ is nonempty. Then $- \Delta_{{\rm K}, \Gamma, q, \cB}$ acts as
\begin{align*}
 \big( - \Delta_{{\rm K}, \Gamma, q, \cB} f \big)_e = - f_e'' + q_e f_e \qquad \text{on each edge}~e \in \cE,
\end{align*}
and its domain consists of all $f \in \widetilde H^2 (\Gamma) \cap H^1 (\Gamma)$ such that
\begin{align*}
 \begin{pmatrix} \partial_\nu f (v_1) \\ \vdots \\ \partial_\nu f (v_d) \end{pmatrix} = \big( \widehat D - B^\top \widehat L^{-1} B \big) \begin{pmatrix} f (v_1) \\ \vdots \\ f (v_d) \end{pmatrix},
\end{align*}
where we have written
\begin{align*}
 \Lambda_{q, \cB} = \begin{pmatrix} \widehat D & - B^\top \\ - B & \widehat L \end{pmatrix},
\end{align*}
 in block matrix form with respect to the decomposition of $\cV$ into $\cB$ and $\cV \setminus \cB$. Moreover,
\begin{align*}
 \dim \ker \big( - \Delta_{{\rm K}, \Gamma, q, \cB} \big) = \dim\ker S_\cB^* = b.
\end{align*}
\end{proposition}

Next, as an application of the abstract Theorem \ref{thm:minMaxKrein}, we obtain the following variational characterisation for the eigenvalues of $- \Delta_{{\rm K}, \Gamma, q, \cB}$.

\begin{theorem}\label{thm:minMaxKreinB}
If Hypothesis \ref{hyp} is satisfied and $\cB \subset \cV$ is nonempty, then the spectrum of $- \Delta_{{\rm K}, \Gamma, q, \cB}$ is purely discrete, and the positive eigenvalues 
\begin{align*}
 \lambda_1^+ \big( - \Delta_{{\rm K}, \Gamma, q, \cB} \big) \leq \lambda_2^+ \big( - \Delta_{{\rm K}, \Gamma, q, \cB} \big) \leq \dots
\end{align*}
of $- \Delta_{{\rm K}, \Gamma, q, \cB}$, counted with multiplicities, satisfy
\begin{align*}
 \lambda_j^+ \big( - \Delta_{{\rm K}, \Gamma, q, \cB} \big) = \min_{\substack{F \subset \dom S_\cB \\ \dim F = j}} \max_{\substack{f \in F \\ f \neq 0}} \frac{\int_\Gamma \left|- f'' + q f\right|^2 \dd x}{\int_\Gamma |f'|^2 \dd x + \int_\Gamma q |f|^2 \dd x}
\end{align*}
for all $j \in \N$.
\end{theorem}

Analogously to Theorem \ref{thm:resDiffGraph}, one may express the resolvent differences of $- \Delta_{{\rm K}, \Gamma, q, \cB}$ with the Friedrichs extension of $S_\cB$ and the perturbed standard Laplacian. In particular, one gets the following.

\begin{theorem}\label{thm:resDiffB}
Assume that Hypothesis \ref{hyp} is satisfied and that $\cB \subset \cV$ is nonempty. Then
\begin{align*}
 \dim \ran \Big[ \big(- \Delta_{{\rm K}, \Gamma, q, \cB} - \lambda \big)^{-1} & - \big(- \Delta_{{\rm st}, \Gamma, q} - \lambda \big)^{-1} \Big] \\
 & = \dim \ran \Lambda_{q, \cB} = \begin{cases} b - 1 & \text{if}~q = 0~\text{identically}, \\ b, & \text{else}, \end{cases}
\end{align*}
where $b = |\cB|$.
\end{theorem}

In particular, in the potential-free case, if $b = |\cB| = 1$, then $- \Delta_{{\rm K}, \Gamma, \cB}$ equals the standard Laplacian. As a consequence of either Theorem \ref{thm:resDiffB} or Theorem \ref{thm:minMaxKreinB}, we get, analogously to \eqref{eq:standardKreinV},
\begin{align*}
 \lambda_{j + 1} \big(- \Delta_{{\rm st}, \Gamma} \big) \leq \lambda_j^+ \big(- \Delta_{{\rm K}, \Gamma, \cB} \big) = \lambda_{j + b} \big(- \Delta_{{\rm K}, \Gamma, \cB} \big), \quad j \in \N,
\end{align*}
in the case without potential. 

The surgery principles of Section \ref{sec:surgery} remain valid for the (positive) eigenvalues of the operator $- \Delta_{{\rm K}, \Gamma, q, \cB}$, provided that all vertices involved in the surgery operations belong to $\cB$; we leave it to the reader to formulate and prove the corresponding results. Instead we formulate a related result which deals with the transition between standard and ``Krein vertex conditions''.

\begin{theorem}\label{thm:addstandardB}
Let Hypothesis \ref{hyp} be satisfied. Moreover, let $\widetilde \cB \subset \cB \subset \cV$ be sets of size $b = |\cB|$ and $\widetilde b = |\widetilde \cB|$, respectively, and let $k := b - \widetilde b$. Then for
\begin{align*}
 & \lambda_j^+ := \lambda_j^+ \big( - \Delta_{{\rm K}, \Gamma, q, \cB} \big), \quad  \widetilde \lambda_j^+ := \lambda_j^+ \big( - \Delta_{{\rm K}, \Gamma, q, \widetilde \cB} \big), \\
 & \lambda_j := \lambda_j \big( - \Delta_{{\rm K}, \Gamma, q, \cB} \big), \quad \,\,\,  \widetilde \lambda_j := \lambda_j \big( - \Delta_{{\rm K}, \Gamma, q, \widetilde \cB} \big),
\end{align*}
the following statements hold:
\begin{enumerate}
 \item the positive eigenvalues satisfy
 \begin{align}\label{eq:inclusion+}
  \widetilde \lambda_j^+ \leq \lambda_j^+ \leq \widetilde \lambda_{j + k}^+ \leq \lambda_{j + k}^+, \quad j \in \N;
 \end{align}
 \item the eigenvalues (counting ground states) satisfy
 \begin{align}\label{eq:inclusion}
  \lambda_j \leq \widetilde \lambda_j \leq \lambda_{j + k} \leq \widetilde \lambda_{j + k}, \quad j \in \N.
 \end{align}
\end{enumerate}
\end{theorem}

\begin{proof}
If we denote by $S$ and $\widetilde S$ the symmetric operators defined in \eqref{eq:SB} for the vertex subsets $\cB$ and $\widetilde B$ respectively, then $\widetilde \cB \subset \cB$ implies the operator inclusion $S \subset \widetilde S$. Moreover, it is easy to see that $\dom S$ has co-dimension $k = b - \widetilde b$ in $\dom \widetilde S$. Therefore \eqref{eq:inclusion+} follows directly from Theorem \ref{thm:interlacingAbstract}. Using the fact that the perturbed Krein Laplacians for $\cB$ and $\widetilde \cB$ have respectively $b$ and $\widetilde b$ linearly independent functions in their kernels, \eqref{eq:inclusion} follows from \eqref{eq:inclusion+}.
\end{proof}

As a simple consequence, for any nonempty $\cB \subset \cV$, we have
\begin{align*}
\lambda_j^+ \big(- \Delta_{{\rm st}, \Gamma, q} \big) \leq \lambda_j^+ \big(- \Delta_{{\rm K}, \Gamma, q, \cB} \big) \leq \lambda_j^+ \big(- \Delta_{{\rm K}, \Gamma, q} \big)
\end{align*}
as well as
\begin{align*}
 \lambda_j \big(- \Delta_{{\rm K}, \Gamma, q} \big) \leq \lambda_j \big(- \Delta_{{\rm K}, \Gamma, q, \cB} \big) \leq \lambda_j \big(- \Delta_{{\rm st}, \Gamma, q} \big),
\end{align*}
for all $j \in \N$.

\begin{ack}
The authors are grateful to Fritz Gesztesy for his comments on the literature. J.R.\ acknowledges financial support by grant no.\ 2018-04560 of the Swedish Research Council (VR).
\end{ack}

\end{document}